\documentclass[12pt]{amsart}

\usepackage{mathtools}
\usepackage{amsmath,amsthm,amsfonts,amssymb,amscd}
\usepackage{mathtools, verbatim, tikz-cd, mathrsfs}
\usetikzlibrary{intersections}
\usetikzlibrary{calc} 
\usepackage{enumerate}
\usepackage[shortlabels]{enumitem}
\usepackage{placeins}

\newcommand{\Z}{\mathbb{Z}}
\newcommand{\R}{\mathbb{R}}

\newcommand{\loc}{_{\textnormal{loc}}}

    \newcommand{\Jcu}{\widehat J^{cu}}
    \newcommand{\Jsc}{\widehat J^{sc}}
    \newcommand{\Jrsc}{\widehat J^{rsc}}

    \newcommand{\Es}{E^s}
    \newcommand{\Ec}{E^c}
    \newcommand{\Eu}{E^u}

    \newcommand{\ti}{\times}

    \newcommand{\Hcal}{\mathcal{H}}

    \newcommand{\subof}{\subset}
    \newcommand{\sans}{\setminus}
    \newcommand{\invn}{^{-n}}

    \newcommand{\bbR}{\mathbb{R}}

    \newcommand{\bbZ}{\mathbb{Z}}

    \newcommand{\Wu}{\mathcal{W}^{u}}
    \newcommand{\Ws}{\mathcal{W}^{s}}
    \newcommand{\Wr}{\mathcal{W}^{r}}
    \newcommand{\Wrloc}{\mathcal{W}^{r}_{\text{loc}}}
    \newcommand{\Wscu}{\mathcal{W}^{scu}}

    \newcommand{\al}{\alpha}
    
    \newcommand{\sig}{\sigma}
    \newcommand{\lam}{\lambda}
    \newcommand{\inv}{^{-1}}

    \newcommand{\bt}{\beta}

    \newcommand{\rcov}{r_{\text{cov}}}
    \newcommand{\dcap}{d_{\text{cap}}}
    \newcommand{\dist}{\operatorname{dist}}

    \newcommand{\del}{\partial}

    \newcommand{\qandq}{\quad \text{and} \quad}

\newtheorem{theo}{Theorem}[section]
\newtheorem{lema}[theo]{Lemma}
\newtheorem{lemma}[theo]{Lemma}

\newtheorem{cor}[theo]{Corollary}
\newtheorem{prop}[theo]{Proposition}

\theoremstyle{remark}

\newtheorem*{notation} {\textbf{Notation}}

\title{Ergodicity of Partially Hyperbolic Endomorphisms}
\thanks{This research is partly supported by the Australian Research Council and by an Australian Mathematical Society Lift-Off Fellowship}
\author{Andy Hammerlindl\and Audrey Tyler}

\usepackage[hyphens]{url}
\usepackage[hidelinks]{hyperref} 

\begin{document}
\begin{abstract}
We prove that for volume preserving, partially hyperbolic, center bunched endomorphisms with constant Jacobian, essential accessibility
implies ergodicity.
\end{abstract}

\maketitle

\section{Introduction}

The history of the study of smooth ergodic theory goes
back at least to the work of E. Hopf in the 1930s,
who developed techniques to show that the geodesic flow of a negatively
curved surface is ergodic \cite{EHopf39}.
His argument, now called the Hopf argument,
has been generalised considerably since then
and applies in some form to many families of volume-preserving dynamical
systems.

Anosov proved in the 1960s that what are now called Anosov
dynamical systems have absolutely continuous stable and unstable folations.
This allowed him to prove that volume-preserving Anosov systems are ergodic \cite{Anosov97}.
Brin and Pesin developed the notion of a partially hyperbolic system
and could establish in certain cases that these systems are ergodic \cite{BP74}.
Grayson, Pugh, and Shub generalised these techniques,
proving that there are $C^2$-open families of
volume-preserving partially hyperbolic diffeomorphisms which are ergodic \cite{GraysonPughShub1994}. 
This further led Pugh and Shub to develop the Pugh-Shub conjectures
which, roughly speaking, state that most partially hyperbolic diffeomorphisms
have a geometric property called accessibility and that
any accessible volume preserving system is ergodic \cite{PughShub96}.

The Pugh-Shub conjectures as stated in their full generality
remain as open questions, but they have been established in many cases
under additional assumptions.
In particular, \cite{RHRHU08}
proved the conjectures under the assumption that the
center is one-dimensional, and Burns and Wilkinson proved that
accessibility together with a property called center bunching
imply that a volume-preserving partially hyperbolic diffeomorphism is ergodic \cite{BW10}.

The advances described above all concern invertible dynamical systems,
either flows or diffeomorphisms. In this paper, we consider the ergodic
properties of non-invertible partially hyperbolic systems.
Mañe, Pugh \cite{Mane}, and Przytycki \cite{przytycki}
developed the notion of an Anosov endomorphism, a self map $f\colon M\to M$
with well-defined stable and unstable directions
$\Eu$ and $\Es$ along each orbit.
They analysed the topological properties of these systems,
such as structural stability.
Partial hyperbolicity has also been generalised, leading to the notion
of a partial hyperbolic endomorphism.
Much of the study of partial hyperbolic endomorphism
has focused on the two-dimensional case
where each orbit has a center direction $\Ec$ and
an expanding unstable direction $\Eu$, each one-dimensional.
Topological properties, such as a classification up to leaf conjugacy
have been proved \cite{Layne-Andy}. 
Ergodic properties in dimension two have also been proved, but using
techniques very different to the Hopf argument.
We do not attempt to give a comprehensive list of references here,
but two relevant results are \cite{Tsujii} and \cite{LDS_Limit_theorems}.

In this paper, we state and prove what we view as the most natural
reformulation of the ergodicity result of Burns and Wilkinson 
in the setting of partially hyperbolic endomorphisms.
In particular, we do not place any assumptions on the dimensions of
the stable, center, or unstable directions.

We first state our main theorem and then later in this introduction
explain how the properties listed as hypotheses in the theorem
are defined in the non-invertible setting.

\begin{theo}[Main]
\label{main}
Let $f$ be a $C^2$, volume-preserving, partially hyperbolic, center bunched endomorphism with constant Jacobian. 
If $f$ is essentially accessible, then $f$ is ergodic.
\end{theo}

The Pugh–Shub conjectures, originally stated for diffeomorphisms, can also be formulated for endomorphisms.
This theorem addresses, in the non-invertible setting, the part of the Pugh–Shub conjectures that states that accessibility implies ergodicity.
We also believe that most partially hyperbolic endomorphisms are accessible.

A significant prior result on
endomorphism conjectures was established by He in 2017 \cite{Baolin};
who proved that most partially hyperbolic, volume-preserving endomorphisms are accessible
under the assumption that the center direction is one-dimensional.
Together with the results of He, our work proves the conjectures in the context of noninvertible systems, assuming that the center is one-dimensional and the endomorphism has a constant Jacobian.

As a specific example, He established stable accessibility for the system $g\colon \mathbb{T}^2\to \mathbb{T}^2$,
\[
g(x,y)=(2x,y+\lambda \sin (2\pi x)),\quad \lambda\neq 0.
\]
Note $g$ has a constant Jacobian of 2. Given that every partially hyperbolic endomorphism with a one-dimensional center is center bunched, $g$ is thus center bunched. Our main result shows that $g$ is ergodic and stably ergodic within the space of partially hyperbolic endomorphisms with a one-dimensional center and constant Jacobian.

\bigskip

In the result of Burns and Wilkinson, they assume that the partially hyperbolic diffeomorphism preserves a measure equivalent to the Lebesgue measure in the manifold $M$.
If this measure is given by a smooth volume form, 
then the measure-preserving assumption implies that the Jacobian of $f$
is constant and equal to one.
In our setting, we assume that the Jacobian is constant and equal to
the degree of $f$ as a covering map
from $M$ to itself.

In \cite{BW10}, Burns and Wilkinson prove the Pugh-Shub conjecture that states that accessibility implies ergodicity for $C^2$, partially hyperbolic, volume-preserving diffeomorphisms under the additional assumption that the diffeomorphism is center bunched.
Similarly,
since we are adapting their work,
it is natural to add the hypothesis of center bunching to our setting.

\bigskip

Let $M$ be a compact Riemannian manifold. An \textit{endomorphism} is a self-covering map $f\colon M\to M$ that is a local diffeomorphism. 
We define the \textit{space of orbits}\index{space of orbits} of the system $(M,f)$ in the following way $$M_f\coloneqq\{\{x_n\}_{n\in \mathbb{Z}}\in M^\mathbb{Z}: f(x_n)=x_{n+1}\}.$$ In other terms, a sequence is an element of $M_f$ if it is a full orbit of $f$.
In this space, we define \textit{the natural projection} $\pi:M_f\to M, \{x_n\}_{n\in \mathbb{Z}} \to x_0$ and the \textit{left shift} $\sigma:M_f\to M_f$. By definition $f(x_n)=x_{n+1}$ for all $\{x_n\}_{n\in \mathbb{Z}}\in M_f$, then $\pi\circ\sigma=f\circ\pi$.

Define $TM_f$, the tangent bundle over $M_f$, as $$TM_f\coloneqq \bigcup_{x\in M_f}\, \bigcup_{v\in T_{\pi(x)}M}\, (x,v).$$
In other terms, an element of $TM_f$ is a pair $(x,v)$ where $x\in M_f$ and $v\in T_{\pi(x)}M$.
The derivative of $\sigma$ is defined as $D\sigma\colon TM_f \to TM_f$, such that $D\sigma (x,v)=(\sigma(x),D_{\pi(x)}f(v))$. Since $\sigma$ is an invertible function and $D_{\pi(x)}f$ is an invertible function for any $x\in M_f$, it follows that $D\sigma$ is an invertible function. \par 
An endomorphism $f$ is \textit{partially hyperbolic} if the following conditions hold. 
There is a non-trivial splitting of the tangent bundle over $M_f$ into three subbundles, 
$$TM_f=E^s\oplus E^c\oplus E^u,$$
that is invariant under the derivative map $D\sigma$. 
Further, there is a Riemannian metric on $TM$ such that when lifted to $TM_f$ we can choose continuous positive functions 
${\lambda^{s},\lambda^{c},\hat\lambda^{c},\lambda^{u}\colon M_f\to\R}$ with $\lambda^s,\lambda^u<1$ and $$\lambda^s<\lambda^c<{(\hat\lambda^c)}^{-1}<{(\lambda^u)}^{-1} $$ such that, for any unit vector $v\in T_pM_f$,
\begin{alignat*}{4}
&\|D_p\sigma v\|&& <{\lambda^s(p)} &\qquad \hbox{if }v\in E^s(p),\\
{\lambda^c(p)}<&\|D_p\sigma v\|&& <{\hat\lambda^c(p)}^{-1} &\qquad \hbox{if }v\in E^c(p),\\
{\lambda^u(p)}^{-1} <&\|D_p\sigma v\|&& &\qquad \hbox{if }v\in E^u(p).
\end{alignat*}

A partially hyperbolic endomorphism $f$ is \textit{center bunched} if the functions $\lambda^{s}$, $\lambda^{c}$, $\hat\lambda^{c}$, and $\lambda^{u}$ can be chosen so that 
\begin{equation*}
\label{eq: center bunched}
\lambda^s<\lambda^c\cdot\hat\lambda^c\quad\text{and}\quad\lambda^u<\lambda^c\cdot\hat\lambda^c.
\end{equation*}

By definition, the subbundles $E^s$ and $E^{cs}$ at $x_0$ are only determined by the derivative $Df$ at $x_n=f^n(x_0)$ for $n\geq 0$.
Since these subbundles depend only on the forward orbit of $x_0 \in M$ and not the full orbit of $\{x_n\}_{n\in \mathbb{Z}}$, we have the property that $E^s(y),E^{cs}(y)\subset T_{x_0}M$ is constant for all $y$ such that $y_0=x_0$.

For a general partially hyperbolic endomorphism,
stable manifolds and unstable manifolds are fundamentally different. 
As in the Anosov case \cite{BP74}, the stable bundle $E^s\subset TM$ of a partially hyperbolic endomorphism is tangent to a foliation $\mathcal{W}^s$ on $M$.
This is illustrated in Figure \ref{fig:StableProjection}.
In general, unstable manifolds do not form a foliation on $M$.
For partially hyperbolic endomorphisms, we have to study each backward orbit of a given point.
Similarly to the work done on Anosov endomorphisms \cite{przytycki} to construct the invariant manifolds, we adapt these arguments to the partially hyperbolic case, and we obtain the unstable foliation on the space of orbits. 
Thus, the unstable bundle \( E^u \) of a partially hyperbolic endomorphism is tangent to a foliation on the space of orbits which we denote by \( \mathcal{W}^u \).
\begin{figure}[ht]
    \centering
\begin{tikzpicture}
\draw[fill=black!20, opacity=.5, draw=none] (-1,-1) -- (3,-1) -- (3,3) -- (-1,3) -- (-1,-1);
\draw[fill=black!20, draw=none] (0,0) -- (2,0) -- (2,2) -- (0,2) -- (0,0);
\node[above] at (1,3) {Neighbourhood of $x_0$ in $M$};
\node[above] at (10,3) {Neighbourhood of $x$ in $M_f$};
\draw[color=blue, thick] (1,0) .. controls +(-0.6,0) and  +(0.5,0) .. node[pos=0, below] {$\mathcal{W}^s(x_0)$} node[midway] (x) {} node[pos=.8] (y) {} (1,2) ;
\node[circle,inner sep=1.5pt,fill=black] at (x) {}; 
\node[circle,inner sep=1.5pt,fill=black] at (y) {}; 
\node[left] at (x) {$x_0$};
\node[left] at (y) {$y_0$};
\draw[thick] (0,0) -- (2,0) -- (2,2) -- (0,2) -- (0,0);
\begin{scope}[shift={(9,1)}]
\foreach \x in {
1,
3,
6,
8,
17,
19,
22,
24}
{
\draw[fill=black!20, opacity=.5, draw=none] (0,0,.3*\x) -- (2,0,.3*\x) -- (2,2,.3*\x) -- (0,2,.3*\x) -- (0,0,.3*\x);
\draw[thick] (0,0,.3*\x) -- (2,0,.3*\x) -- (2,2,.3*\x) -- (0,2,.3*\x) -- (0,0,.3*\x);
\draw[color=blue, thick] (1,0,.3*\x) .. controls +(-0.6,0,0) and  +(0.5,0,0) .. node[midway] (x) {} node[pos=.8] (y) {} (1,2,.3*\x) ;
\node[circle,inner sep=1.5pt,fill=black] at (x) {}; 
\node[circle,inner sep=1.5pt,fill=black] at (y) {}; 
}
\draw[color=blue, thick] (1,0,.3*24) .. controls +(-0.6,0,0) and  +(0.5,0,0) .. node[midway] (x1) {} node[pos=.8] (y1) {} (1,2,.3*24) ;
\node[circle,inner sep=1.5pt,fill=black] at (x1) {}; 
\node[circle,inner sep=1.5pt,fill=black] at (y1) {}; 
\node[left] at (x1) {$x$};
\node[left] at (y1) {$y$};
\end{scope}
\draw[thick,->] (5.5,1) -- node[above,midway] {$\pi$} (3.5,1);

\end{tikzpicture}

    \caption{Unlike the unstable foliation in $M_f$, the stable foliation projects down by $\pi$ to the unique unstable foliation on $M$.}
    \label{fig:StableProjection}
\end{figure}
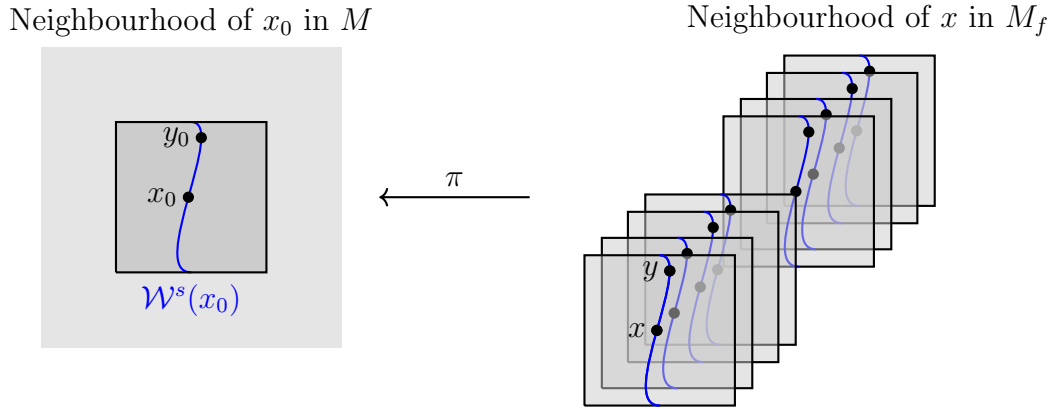

The leaves of the unstable foliation can be projected down to obtain unstable manifolds in the base manifold $M$.
A consequence of not having unique backward orbits is that there can be uncountably many unstable manifolds going through a point, each associated to a particular backward orbit of that point. 
This is illustrated in Figure \ref{fig:multipleunstables}.
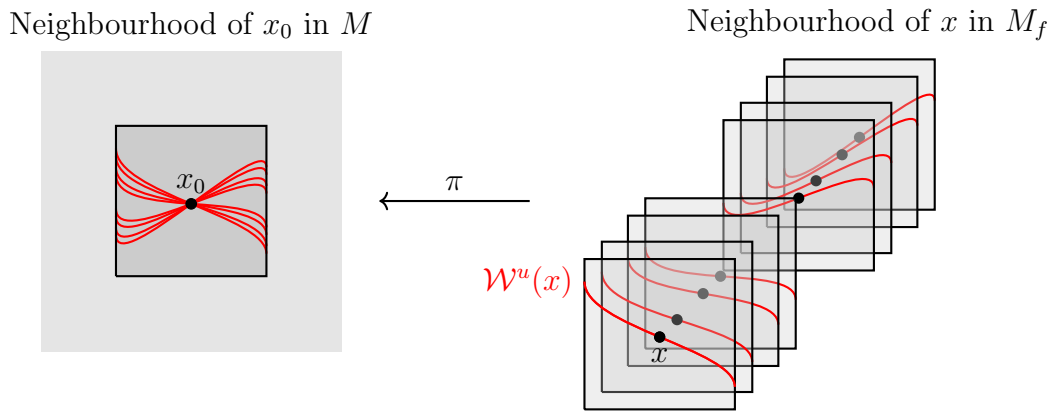
\begin{figure}[ht]
    \centering

\begin{tikzpicture}
\draw[fill=black!20, opacity=.5, draw=none] (-1,-1) -- (3,-1) -- (3,3) -- (-1,3) -- (-1,-1);
\draw[fill=black!20, draw=none] (0,0) -- (2,0) -- (2,2) -- (0,2) -- (0,0);
\node[above] at (1,3) {Neighbourhood of $x_0$ in $M$};
\node[above] at (10,3) {Neighbourhood of $x$ in $M_f$};

\foreach \x in {
1,
3,
6,
8,
17,
19,
22,
24}
{ 
\draw[color=red, thick] (0,{.5+0.05*\x}) .. controls +(0,-0.6,0) and  +(0,0.5,0) .. node[midway] (x) {} (2,1.5-.05*\x) ;
}
\node[above] at (x) {$x_0$};
\filldraw [black] (x) circle (2pt);
\draw[thick] (0,0) -- (2,0) -- (2,2) -- (0,2) -- (0,0);
\begin{scope}[shift={(9,1)}]
\foreach \x in {
1,
3,
6,
8,
17,
19,
22,
24}
{
\draw[fill=black!20, opacity=.3, draw=none] (0,0,.3*\x) -- (2,0,.3*\x) -- (2,2,.3*\x) -- (0,2,.3*\x) -- (0,0,.3*\x);
\draw[color=red, thick] (0,{.5+0.05*\x},.3*\x) .. controls +(0,-0.6,0) and  +(0,0.5,0) .. node[midway] (x) {} (2,1.5-.05*\x,.3*\x) ;
\node[circle,inner sep=1.5pt,fill=black] at (x) {}; 
\draw[thick] (0,0,.3*\x) -- (2,0,.3*\x) -- (2,2,.3*\x) -- (0,2,.3*\x) -- (0,0,.3*\x);
}
\draw[color=red, thick] (0,{.5+0.05*24},.3*24) .. controls +(0,-0.6,0) and  +(0,0.5,0) .. node[midway] (z) {}  node[pos=0, left] {$\mathcal{W}^u(x)$} (2,1.5-.05*24,.3*24);
\node[below] at (z) {$x$} ;
\node[circle,inner sep=1.5pt,fill=black] at (z) {}; 
\end{scope}
\draw[thick,->] (5.5,1) -- node[above,midway] {$\pi$} (3.5,1);
\end{tikzpicture}

    \caption{Multiple unstable manifolds passing through a point \( x_0 \in M \), each corresponding to a different past orbit of \( x_0 \) in the space of orbits.}
    \label{fig:multipleunstables}
\end{figure}

To prove Theorem \ref{main},
we now need to consider four different ``directions''.
There are the $s,$ $c,$ and $u$ directions coming from the splitting
$\Es \oplus \Ec \oplus \Eu$ along each orbit as well as the fiber direction.
We associate the fiber direction with the letter ``$r$'' in order to use it
in subscripts and superscripts.
In particular, for two orbits
$x = \{ x_n \}_{n \in \bbZ}$
and
$y = \{ y_n \}_{n \in \bbZ}$ in $M_f,$
we define sets $\mathcal{W}^r\loc(x)$ and $\mathcal{W}^r(x)$ by
\begin{itemize}
    \item $y \in \mathcal{W}^r\loc(x)$ if and only if $x_0 = y_0$ and
    \item $y \in \mathcal{W}^r(x)$ if and only if $x_n = y_n$ for some $n \in \bbZ.$
\end{itemize}
The set $\mathcal{W}^r(x)$ can be thought of as a strong stable set where the contraction
along the $r$ direction is much stronger than the contraction along the
$s$ direction. We chose the letter ``$r$'' both because it is the last letter
in the American spelling of ``fiber'' and because $r$ comes before $s$
alphabetically and the two directions can both be considered as stable
directions.

\subsection{Accessibility}
In 2000, as noted above, Pugh and Shub \cite{PS2000} proposed the concept of accessibility to study the original Pugh-Shub conjecture in the diffeomorphism case by dividing the conjecture into two subconjectures.
A partially hyperbolic diffeomorphism $f$ is \textit{accessible} if for any two points $x,y\in M$, there is a finite set of points $x=x_0,x_1,\dots,x_n=y$, such that the pair $x_i$, $x_{i+1}$ both lie in either a common stable or unstable manifold. 

In the partially hyperbolic diffeomorphism case, accessibility comes from following paths along two quasi-transverse foliations, the stable and unstable foliations.
Two distinct leaves of the same kind, stable or unstable, do not intersect each other.
This is visualised in Figure \ref{fig:acccomp}.

A partially hyperbolic endomorphism $f$ is \textit{accessible} if, for any two points $x,y\in M$, there is a finite set of points $x=x_0,x_1,\dots,x_n=y$, such that the pair $x_i$, $x_{i+1}$ both lie in either a common stable or unstable manifold. This sequence $x_0,x_1,\dots,x_n$ is called an \textit{$su$-path}. The \textit{accessibility class of }$p\in M$ is the set of all $q\in M$ that can be reached from $p$ along an $su$-path. This is an equivalence relation. It follows that $f$ is accessible if and only if there is a single accessibility class. A partially hyperbolic endomorphism $f$ is \textit{essentially accessible} if every Lebesgue measurable set that is a union of entire accessibility classes has either full or zero measure.

In the endomorphism case, we can change from one unstable leaf to another at a single point.

\subsection{Weak partial hyperbolicity case}
Weakly partially hyperbolic endomorphisms differ mainly from those studied in the rest of this paper in that they do not have a stable direction,
that is,
$\dim(E^s)=0$. In this case, the tangent bundle of the space of orbits has a nontrivial invariant splitting into center and unstable subspaces.
\[
TM_f= E^c\oplus E^u.
\]

To observe that $\dim(E^u)>0$ is needed, we must recall that if we had $\dim(E^u)=0$ and $\dim(E^s)>0$, then we would have a stable foliation on the base manifold $M$. As shown in Figure \ref{fig:StableProjection}, we would have a unique stable leaf through each point. By the definition of foliation, two points are either in the same stable leaf or in disjoint stable leaves.
Hence, two points on different stable leaves would never be connected by stable leaves. So accessibility is not possible in this setting.


In Figure \ref{fig:multipleunstables}, a partially hyperbolic endomorphism is illustrated, where projecting unstable curves from $M_f$ to $M$, each belonging to different orbits, results in an intersection. It is the intersecting unstable manifolds passing through each point that may allow us to recover accessibility. 
Figure \ref{fig:acccomp} shows the different behaviours between accessibility for strong partially hyperbolic diffeomorphisms and weak partially hyperbolic endomorphisms.
While in the diffeomorphism case, accessibility may rise from a sequence of stable and unstable leaves intersecting.
In the weakly partially hyperbolic endomorphism case, accessibility depends only on unstable manifolds defined by different past orbits.

\begin{figure}[h]
    \centering
\begin{tikzpicture}
\begin{scope}[shift={(0,4)}]
    \node[draw,align=left] at (4,1) {Accessibility for \textbf{strong} partially hyperbolic \textbf{diffeomorphisms}};
    \draw[blue, thick] (-.5,-1) .. controls (0,-.5) and (2,-1) .. node[midway,above] {\textit{s}} (2.5,0);
    \draw[red, thick] (0,0) .. controls (.2,-1) and (-.2,-1) .. node[pos=0,above] {\textit{u}} node[pos=.9] (x) {} (0,-2);
    \draw[red, thick] (4,.5) .. controls +(0,-1) and +(0,1) .. node[pos=1,below] {\textit{u}} (8.5,-2);
    \draw[blue, thick] (8.5,0) .. controls +(-1,-1) and +(1,1) .. node[pos=1,below left] {\textit{s}} node[pos=.1] (y) {} (7,-2);

    \draw[red, thick] (2,0.1) .. controls (2.2,-1) and (1.6,-1) .. node[pos=1,below] {\textit{u}} (2,-2);
    \draw[blue, thick] (1.5,-1.5) .. controls +(1,1) and +(-1,-1) .. node[midway,above left] {\textit{s}}  (7,0);
    \fill (x) circle [radius=2pt];
    \node [left] at (x) {$x$};
    \fill (y) circle [radius=2pt];
    \node [left] at (y) {$y$};
\end{scope}

\node[draw,align=left] at (4,1) {Accessibility for \textbf{weak} partially hyperbolic \textbf{endomorphisms}};
\draw[red, thick] (-.5,-1) .. controls (0,-.5) and (4,-1) .. node[near end, below left] {\textit{u}} (5,0);
\draw[red, thick] (0,-1.5) .. controls (0,-.5) and (2,-1) .. node[pos=.1] (x0) {} node[pos=1,above] {\textit{u}} (2,0);
\draw[red, thick] (4,.5) .. controls +(0,-1) and +(0,1) .. node[midway,above] {\textit{u}} (8.5,-2);
\draw[red, thick] (8.5,0) .. controls +(-1,-1) and +(1,1) ..  node[pos=1,below left] {\textit{u}} node[pos=.1] (y0) {} (7,-2);

\fill (x0) circle [radius=2pt];
\node[left] at (x0) {$x_0$};

\fill (y0) circle [radius=2pt];
\node[left] at (y0) {$y_0$};

\begin{scope}[shift={(0,-4)}]
    \node[draw,align=left] at (4,1) {Accessibility for \textbf{strong} partially hyperbolic \textbf{endomorphisms}};
\end{scope}
\begin{scope}[shift={(0,-5)}]
\draw[blue,thick] (-.5,-1) .. controls +(0.5,.5) and +(-.2,-.5) .. node[pos=.2] (x) {} node[midway, left] {\textit{s}} (1.5,1);
\draw[blue,thick] (6.5,-1) .. controls +(0.3,.3) and +(-.1,-.25) .. node[pos=.8] (y) {} node[midway, right] {\textit{s}} (8.5,1);
\draw[red,thick] (1,1) .. controls +(0,-.2) and +(0.1,0.1) .. node[near end, above] {\textit{u}} (5,0);
\draw[red,thick] (5,-.75) .. controls +(-0.5,0.5) and +(.5,-.5) .. node[pos=0, below] {\textit{u}} (2.5,.75);
\draw[red,thick] (2,0) .. controls +(1,0) and +(0,0.2) .. node[near start, below] {\textit{u}} (7,-.75);

\fill (x) circle [radius=2pt];
\fill (y) circle [radius=2pt];
\node[left] at (y) {$w_0$};
\node[below] at (x) {$t_0$};
\end{scope}
\end{tikzpicture}

    \caption{In the strong partially hyperbolic diffeomorphism case, an \textit{su}-path connects $x$ to $y$ by alternating stable and unstable manifolds. In the weak partially hyperbolic endomorphism case, an \textit{u}-path connects $x_0$ to $y_0$, inside the base manifold $M$, using only unstable curves.
    In the strong partially hyperbolic endomorphism case, an \textit{su}-path connects $t_0$ to $w_0$, inside the base manifold $M$.}
    \label{fig:acccomp}
\end{figure}
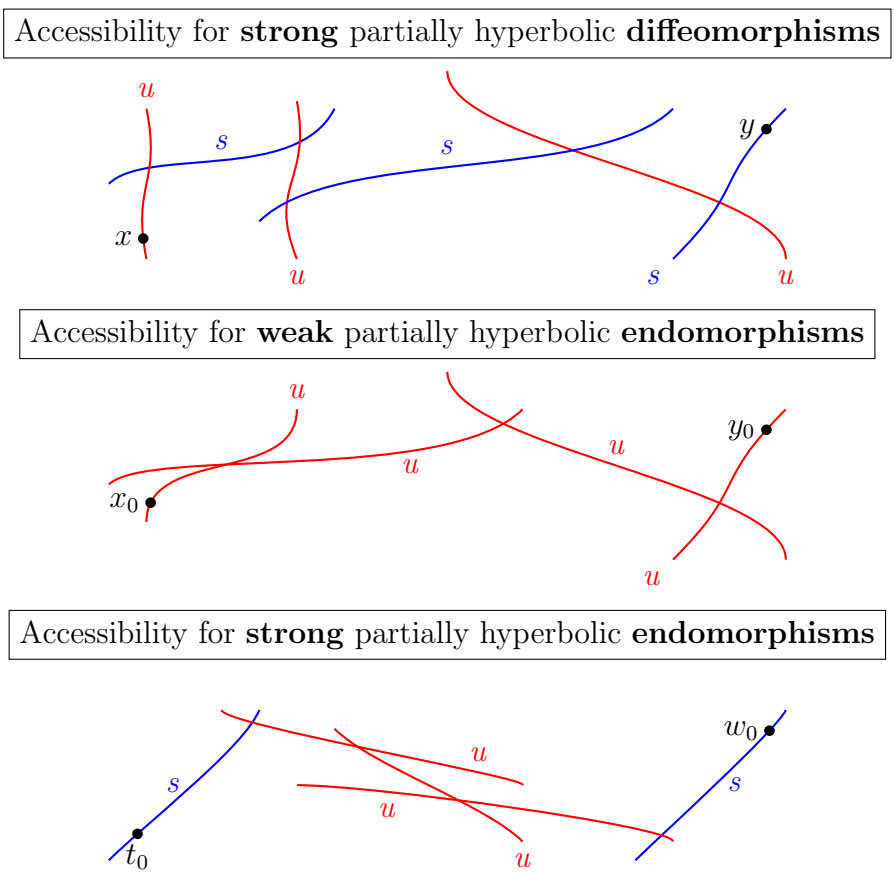

For weakly partially hyperbolic endomorphisms,
accessibility relies only on unstable manifolds. 

Before moving on to our more general case, we want to take a last look at accessibility for weak partially hyperbolic endomorphisms. Figure \ref{figaccmf} shows that while on the base manifold a \textit{u}-path might appear merely as the concatenation of unstable curves contained in intersecting unstable manifolds, the space of orbits reveals what is underlying. 
On the base manifold at an intersection, we change to another unstable curve.
On the space of orbits, since the unstable manifolds are now an unstable foliation, we must use the \textit{r}-direction to move between them. At each intersection, we are changing the past orbit of the point and then moving along the unstable curve inside the unstable leaf through this new orbit.

We could define accessibility in the space of orbits in terms of \textit{u}-paths and the \textit{r}-direction. This definition is equivalent to the one above in this section for weakly partially hyperbolic endomorphisms. To extend the definition to the strong partially hyperbolic endomorphism, one would only need to include \textit{su}-paths and the \textit{r}-direction, a definition which again is equivalent to the main definition of accessibility used in this paper.

\begin{figure}[ht]
    \centering
\begin{tikzpicture}
\draw[name path=ux,white] (-.5,-1) .. controls (0,-.5) and (4,-1) .. (5,0);
\draw[name path=sx,red, thick] (0,-1.5) .. controls (0,-.5) and (2,-1) .. node[pos=.1] (x0) {} node[pos=0,below] {\textit{u}} (2,0);
\draw[name path=rojo,red, thick] (4,.5) .. controls +(0,-1) and +(0,1) .. node[midway,above] {\textit{u}} (8.5,-2);
\draw[name path=cxprime,color=white] (8.5,0) .. controls +(-1,-1) and +(1,1) ..  (7,-2);

\fill [name intersections={of=ux and sx, by=x}] (x) circle [radius=2pt];
\fill [name intersections={of=rojo and cxprime, by=xprime}] (xprime) circle [radius=2pt];
\node [anchor=north] at (x) {$x^2$};
\node [anchor=north] at (xprime) {$x^6$};
\draw[name path=s3,densely dashed, color=blue!30!black,thick] (x) -- node[midway,right] {\textit{r}} +(.5,1.26);    
\draw[->,thick] (4,-1.5) to node[anchor=west] {$\pi$} (4,-3);

\fill (x0) circle [radius=2pt];
\node[left] at (x0) {$x^1$};

\begin{scope}[shift={(.5,1.25)}]
\draw[name path=u2,red, thick] (-.5,-1) .. controls (0,-.5) and (4,-1) .. node[midway,above] {\textit{u}} (5,0);    
\fill [name intersections={of=u2 and s3, by=x2}] (x2) circle [radius=2pt];
\node [anchor=south] at (x2) {$x^3$};
\draw[name path=n,white] (4,.5) .. controls +(0,-1) and +(0,1) .. (8.5,-2);
\fill [name intersections={of=u2 and n, by=x3}] (x3) circle [radius=2pt];
\draw[name path=u4,color=red, thick] (8.5,0) .. controls +(-1,-1) and +(1,1) .. node[pos=0,above right] {\textit{u}} node[pos=.1] (x8) {} (7,-2);
\fill [name intersections={of=u4 and n, by=m}] (m) circle [radius=2pt];

\fill (x8) circle [radius=2pt];
\node[right] at (x8) {$x^8$};

\end{scope}
\fill [name intersections={of=ux and rojo, by=x4}] (x4) circle [radius=2pt];
\draw[name path=s1,densely dashed, color=blue!30!black,thick] (x3) -- node[midway,right] {\textit{r}} (x4);
\draw[name path=s1,densely dashed, color=blue!30!black,thick] (xprime) -- node[midway,right] {\textit{r}} (m);

\begin{scope}[shift={(0,-4)}]
\draw[name path=ux,red, thick] (-.5,-1) .. controls (0,-.5) and (4,-1) .. node[near end, above left] {\textit{u}} (5,0);
\draw[name path=sx,red, thick] (0,-1.5) .. controls (0,-.5) and (2,-1) .. node[pos=.1] (x0) {} node[pos=0,below] {\textit{u}} (2,0);
\draw[name path=rojo,red, thick] (4,.5) .. controls +(0,-1) and +(0,1) .. node[midway,above] {\textit{u}} (8.5,-2);
\draw[name path=cxprime,color=red, thick] (8.5,0) .. controls +(-1,-1) and +(1,1) .. node[pos=0,above right] {\textit{u}} node[pos=.1] (y0) {} (7,-2);

\fill [name intersections={of=ux and sx, by=x}] (x) circle [radius=2pt];
\fill [name intersections={of=ux and rojo, by=z}] (z) circle [radius=2pt];
\fill [name intersections={of=rojo and cxprime, by=xprime}] (xprime) circle [radius=2pt];
\node [above left] at (x) {$\pi(x^2)=\pi(x^3)$};
\node [right] at (xprime) {$\pi(x^6)=\pi(x^7)$};
\node [anchor=west] at (x3) {$x^4$};
\node [anchor=north] at (x4) {$x^5$};
\node [anchor=west] at (m) {$x^7$};
\node [anchor=south west] at (z) {$\pi(x^4)=\pi(x^5)$};
\fill (x0) circle [radius=2pt];
\node[left] at (x0) {$\pi(x^1)=x_0$};

\fill (y0) circle [radius=2pt];
\node[right] at (y0) {$\pi(x^8)=y_0$};
\end{scope}
\end{tikzpicture}
\caption{The projection $\pi$ relates the notions of accessibility on the space of orbits $M_f$ and on the manifold $M$. Note that this is the weakly partially hyperbolic endomorphism case.}
\label{figaccmf}
\end{figure}
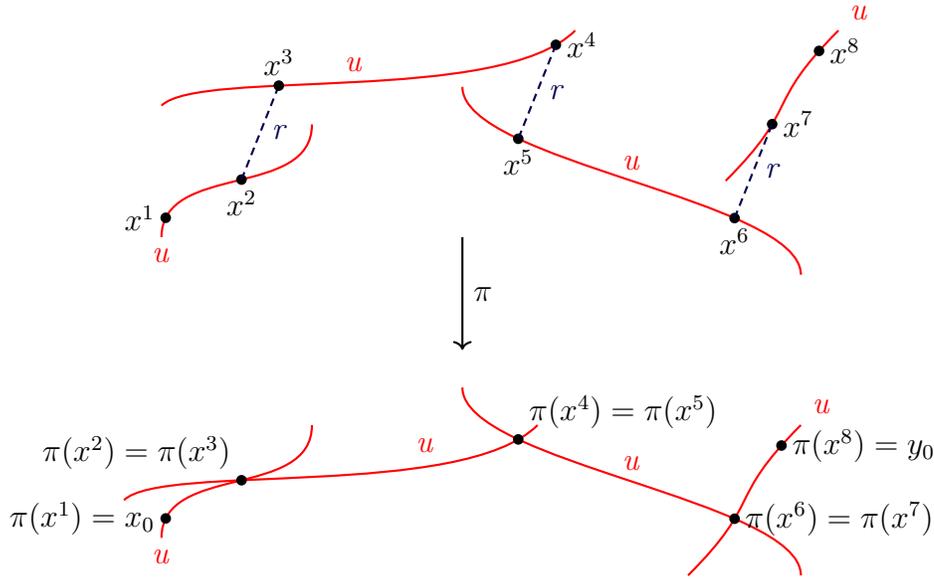

Weak partially hyperbolic endomorphisms were the original setting in which we studied this problem, and as discussed above, we can define accessibility to state our result with $\dim(E^s)=0$.
All proofs, including the main result, can naturally be adapted to that case.

\subsection{Accessibility in strong partially hyperbolicity case}
When both ${\dim(E^s)>0}$ and $\dim(E^u)>0$ accessibility takes the characteristics of accessibility for strong partially hyperbolic diffeomorphisms and weak partially hyperbolic endomorphisms. We have a stable foliation such that through each point we have a unique stable leaf illustrated in Figure \ref{fig:StableProjection}.
In the
general
case, we also have intersecting unstable manifolds through a point. 

In the general setting of Theorem \ref{main},
it is possible for a \textit{su}-path to go from one unstable manifold through a point to another. This means that when the path is lifted to the space of orbits $M_f$, there is a path that has legs in all of the stable, unstable, and the fibre or ``$r$'' directions. This behaviour is illustrated in Figure \ref{fig:endostrong}.

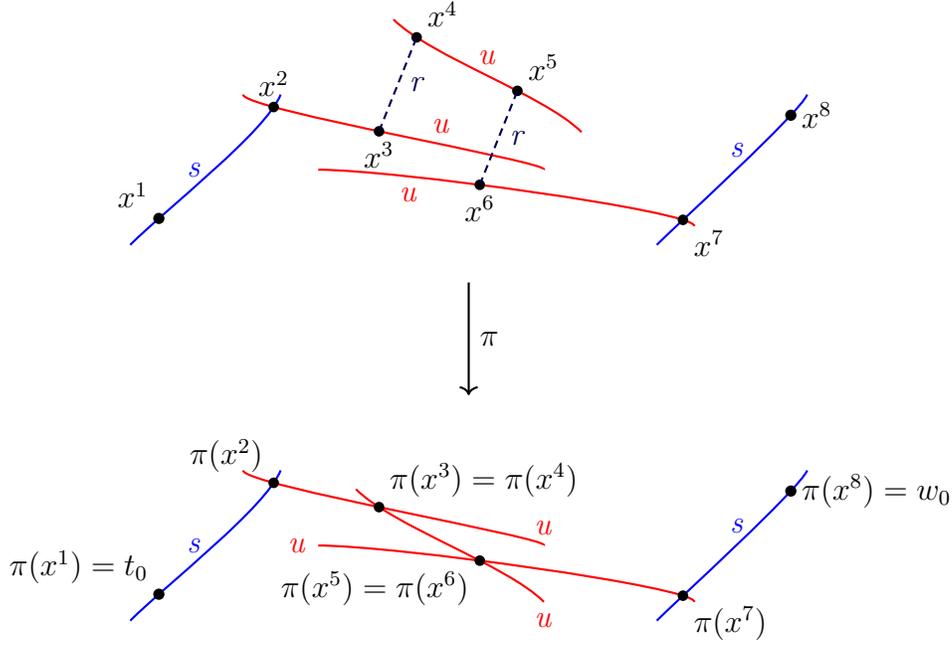
\begin{figure}
    \centering
\begin{tikzpicture}
\draw[name path=s11,blue,thick] (-.5,-1) .. controls +(0.5,.5) and +(-.2,-.5) .. node[pos=.2] (x1) {} node[midway, left] {\textit{s}} (1.5,1);
\draw[name path=s22,blue,thick] (6.5,-1) .. controls +(0.3,.3) and +(-.1,-.25) .. node[pos=.8] (x88) {} node[midway,above] {\textit{s}} (8.5,1);
\draw[name path=u11,red,thick] (1,1) .. controls +(0,-.2) and +(0.1,0.1) .. node[pos=.6, above] {\textit{u}} (5,0);
\draw[name path=u22,draw=none,thick] (5,-.75) .. controls +(-0.5,0.5) and +(.5,-.5) ..  (2.5,.75);
\draw[name path=u33,red,thick] (2,0) .. controls +(1,0) and +(0,0.2) .. node[near start, below] {\textit{u}} (7,-.75);
\node[circle,inner sep=1.5pt,fill=black] at (x1) {}; 
\node[above left] at (x1) {$x^1$};
\fill [name intersections={of=s11 and u11, by=x2}] (x2) circle [radius=2pt];
\node[above] at (x2) {$x^2$};
\fill [name intersections={of=u11 and u22, by=x33}] (x33) circle [radius=2pt];
\node[below] at (x33) {$x^3$};
\fill [name intersections={of=u22 and u33, by=x66}] (x66) circle [radius=2pt];
\node[below] at (x66) {$x^6$};
\fill [name intersections={of=u33 and s22, by=x77}] (x77) circle [radius=2pt];
\node[below right] at (x77) {$x^7$};
\node[circle,inner sep=1.5pt,fill=black] at (x88) {}; 
\node[right] at (x88) {$x^8$};
\begin{scope}[shift={(.5,1.25)}]
\draw[name path=u111,draw=none,thick] (1,1) .. controls +(0,-.2) and +(0.1,0.1) ..  (5,0);
\draw[name path=u222,red,thick] (5,-.75) .. controls +(-0.5,0.5) and +(.5,-.5) .. node[midway, above] {\textit{u}} (2.5,.75);
\draw[name path=u333,draw=none,thick] (2,0) .. controls +(1,0) and +(0,0.2) .. (7,-.75);
\fill [name intersections={of=u111 and u222, by=x444}] (x444) circle [radius=2pt];
\fill [name intersections={of=u222 and u333, by=x555}] (x555) circle [radius=2pt];
\node[above right] at (x444) {$x^4$};
\node[above right] at (x555) {$x^5$};
\end{scope}
\draw[name path=s3,densely dashed, color=blue!30!black,thick] (x33) -- node[midway,right] {\textit{r}} (x444);    
\draw[name path=s3,densely dashed, color=blue!30!black,thick] (x66) -- node[midway,right] {\textit{r}} (x555);    
\draw[->,thick] (4,-1.5) to node[anchor=west] {$\pi$} (4,-3);
\begin{scope}[shift={(0,-5)}]
\draw[name path=s1,blue,thick] (-.5,-1) .. controls +(0.5,.5) and +(-.2,-.5) .. node[pos=.2] (x) {} node[midway, left] {\textit{s}} (1.5,1);
\draw[name path=s2,blue,thick] (6.5,-1) .. controls +(0.3,.3) and +(-.1,-.25) .. node[pos=.8] (y) {} node[midway,above] {\textit{s}} (8.5,1);
\draw[name path=u1,red,thick] (1,1) .. controls +(0,-.2) and +(0.1,0.1) .. node[pos=1, above] {\textit{u}} (5,0);
\draw[name path=u2,red,thick] (5,-.75) .. controls +(-0.5,0.5) and +(.5,-.5) .. node[pos=0, below] {\textit{u}} (2.5,.75);
\draw[name path=u3,red,thick] (2,0) .. controls +(1,0) and +(0,0.2) .. node[pos=0, left] {\textit{u}} (7,-.75);

\fill [name intersections={of=s1 and u1, by=x2}] (x2) circle [radius=2pt];
\fill [name intersections={of=u1 and u2, by=x3}] (x3) circle [radius=2pt];
\fill [name intersections={of=u2 and u3, by=x4}] (x4) circle [radius=2pt];
\fill [name intersections={of=u3 and s2, by=x5}] (x5) circle [radius=2pt];
\node[circle,inner sep=1.5pt,fill=black] at (y) {}; 
\node[circle,inner sep=1.5pt,fill=black] at (x) {}; 
\node[right] at (y) {$\pi(x^8) =w_0$};
\node[above left] at (x) {$\pi(x^1)=t_0$};
\node[above left] at (x2) {$\pi(x^2)$};
\node[above right] at (x3) {$\pi(x^3)=\pi(x^4)$};
\node[below left] at (x4) {$\pi(x^5)=\pi(x^6)$};
\node[below right] at (x5) {$\pi(x^7)$};
\end{scope}
\end{tikzpicture}
    \caption{An $su$-path in the base manifold $M$ lifted to the space of orbits where it becomes a path with stable, unstable and ``$r$'' segments.}
    \label{fig:endostrong}
\end{figure}

\subsection{Proof outline}
We now briefly outline the structure of the proof of \cite{BW10} and then explain
the major modifications we make to prove our version in the
non-invertible setting.
The proof in \cite{BW10} is a version of the Hopf argument.
In order to show that $f$ is ergodic with respect to the measure $\mu,$
consider a continuous function $\phi \colon M \to \R$ and let
$\phi^s \colon M \to \R$ and $\phi^u\colon M \to \R$ denote the forward and backward
Birkhoff averages respectively.
The Birkhoff Ergodic Theorem implies that $\phi^s$ and $\phi^u$
exist and are equal $\mu$-almost everywhere.
Fix a constant $a \in \bbR$ and define the measurable sets
\[
    A^s \coloneqq \{ x \in M : \phi^s(x) \le a \} 
    \qandq
    A^u \coloneqq \{ x \in M : \phi^u(x) \le a \}.
\]

To prove ergodicity, it is enough to prove for all continuous $\phi$
and all $a \in \R$ that $A^s$ has either full measure or zero measure.
By the Lebesgue Density Theorem,
this is equivalent to showing that the set $X$
of all Lebesgue density points of $A^s$ has full or zero measure.
Burns and Wilkinson show that the set of Lebesgue density points is saturated
in the stable direction; that is, $X$ consists of a union of complete
stable leaves.
Since $A^s$ and $A^u$ are equal mod zero,
$X$ is also the set of Lebesgue density points of $A^u$ and
an analogous argument shows that $X$ is a union of unstable leaves.
Therefore $X$ is a measurable set which is both
$u$-saturated and $s$-saturated, and the assumption of essential accessibility
implies that it must have full or zero measure.

We now look at their proof of stable saturation in more detail.
As in \cite{BW10}, we adopt the notation for the measure $\mu$ (or any measure)
and measurable sets $A$ and $B$ with $\mu(B) > 0$ that the conditional measure is
defined as
\[
    \mu(A : B) \ = \ \frac{\mu(A \cap B)}{\mu(B)}.
\]
A Lebesgue density point for the set $A^s$ is then a point $x \in M$
which satisfies
$\lim_{r \searrow 0} \mu(A^s : B_r(x)) = 1$
where $B_r(x)$ is the ball of radius $r$ centered at $x.$
To show that $A^s$ is $s$-saturated,
Burns and Wilkinson construct for each point $x \in M$ a sequence of
``$cu$-juliennes'' $\Jcu_n(x)$ and show
(using other carefully defined sequences of juliennes) that
\[
    \lim_{r \searrow 0} \mu(A^s : B_r(x)) = 1
    \quad \Leftrightarrow \quad
    \lim_{n \to \infty} m^{cu}(A^s : \Jcu_n(x)) = 1.
\]
Each $cu$-julienne is a small embedded submanifold containing $x,$
has the same dimension as $\Ec \oplus \Eu,$
and is approximately tangent to $\Ec \oplus \Eu.$
The measure $m^{cu}$ is defined using the Riemannian metric on the julienne.
These juliennes are nested,
$\Jcu_1(x) \supset \Jcu_2(x) \supset \Jcu_3(x) \supset \ldots $,
and their diameters shrink exponentially fast with $n.$
In fact, they shrink at different exponential rates in the
center and unstable directions.
If $x$ and $x'$ are nearby points on the same stable manifold,
we can define a stable holonomy, $h^s,$ which takes points on
a $cu$-julienne through $x$ to a $cu$-julienne through $x'$
by travelling along the stable leaves.
They prove an ``internesting'' property:
there is a uniform constant integer $k$ such that
\[
    \Jcu_{n+k}(x') \subof h^s(\Jcu_n(x)) \subof \Jcu_{n-k}(x')
\]
holds for any two points $x$ and $x'$ on the same local stable leaf.
From this, they prove
\[
    \lim_{n \to \infty} m^{cu}(A^s : \Jsc_n(x)) = 1
    \quad \Leftrightarrow \quad
    \lim_{n \to \infty} m^{cu}(A^s : \Jsc_n(x')) = 1
\]
which combines with the earlier equivalence to show that
$X$ is $s$-saturated.

The definition of partial hyperbolicity for a diffeomorphism
is symmetric in the sense that $f$ is partially hyperbolic if and only if
$f \inv$ is partially hyperbolic with the roles of $s$ and $u$ swapped.
This means that an analogous argument can be given to define a sequence of
$sc$-juliennes and use the unstable holonomy to prove that $X$ is $u$-saturated.

\medskip{}

We now outline how we have adapted these techniques to the non-invertible
setting. Let $f \colon M \to M$ now denote a partially hyperbolic endomorphism
which satisfies the assumptions of Theorem \ref{main}.
Recall from above the definitions of $M_f$ the space of orbits of $f$,
the left shift map $\sig$
and the projection $\pi$.
The $f$-invariant measure $\mu$ on $M$ lifts naturally to a $\sig$-invariant
measure $\nu$ on $M_f$ such that $\pi_* \nu = \mu$.

Let $\varphi_0 \colon M \to \bbR$ be a continuous function and
let $\varphi = \varphi_0 \circ \pi \colon M_f \to \bbR$ be its lift to $M_f.$
As $f$ is non-invertible, $\varphi_0$ only has a forward Birkhoff average
$\varphi_0^+ \colon M \to \bbR.$ As the shift map $\sig$ is a homeomorphism of $M_f,$
$\varphi$ has both forward and backward Birkhoff averages,
$\varphi^+, \varphi^- \colon M_f \to \bbR,$ which are equal $\nu$-almost everywhere.
Moreover, $\varphi^+ = \varphi_0^+ \circ \pi.$
Fix a constant $a \in \bbR$ and define
\begin{align*}
    A_0
        &= \{ x_0 \in M : \varphi_0^+(x_0) \le a \},
    \\
    A^{rs}
        &= \{ x \in M_f : \varphi^+(x) \le a \}, \quad \text{and}
    \\
    A^{u}
        &= \{ x \in M_f : \varphi^-(x) \le a \}.
\end{align*}
Note that $A^{rs}$ and $A^{u}$ are equal modulo a set of $\nu$-measure zero
and that $\pi \inv(A_0) = A^{rs}.$
To prove ergodicity, it is enough to show that the Lebesgue density points
of $A_0$ have either full or zero $\mu$-measure in $M.$

Locally, the space of orbits has the topological structure of a manifold
times a Cantor set.
Specifically, if $U \subof M$ is a small open set and $x_0$ is a point in $U,$
then 
$\pi \inv(U) \subof M_f$ is homeomorphic to $U \ti \pi \inv(x_0)$
where the fiber $\pi \inv(x_0)$ is a Cantor set.
Throughout the proof, we abuse notation and treat $\pi \inv(U)$
as if it were equal to $U \ti \pi \inv(x_0).$
Moreover, the assumption of constant Jacobian means that
the measure $\nu$ restricted to $\pi \inv(U) = U \ti \pi \inv(x_0)$
is a product measure coming from the restriction of $\mu$ to $U$
and a Bernoulli measure on the Cantor set $\pi \inv(x_0).$

We will show that the set of Lebesgue density points of $A_0$ is both saturated
by stable leaves, which form a well-defined foliation on $M,$
and by the images of unstable leaves under the projection
$\pi \colon M_f \to M.$
These projections of unstable leaves do not form a foliation on $M,$
and therefore in the non-invertible setting,
the proofs of stable saturation
and unstable saturation are quite different.

\medskip{}

For the stable saturation, consider a small open disk $U \subof M$
and two points $x_0, x_0' \in U$ which are connected by a short stable
segment inside of $U.$
Recall that we have identified $\pi \inv(U)$ with the product space
$U \ti \pi \inv(x_0)$ and that the measure $\nu$ restricted to this set
is a product measure on $U \ti \pi \inv(x_0).$
Since $A^{rs}$ is equal to $A^{u}$ modulo a set of $\nu$-measure zero,
it follows by Fubini's theorem that for almost every copy of $U$
in $U \ti \pi \inv(x_0),$ the sets
$A^{rs} \cap (U \ti x)$ and $A^{u} \cap (U \ti x)$ are equal
almost everywhere on $U \ti x.$
We can then use the \cite{BW10} argument nearly unchanged inside any one of these
``good'' copies of $U$ to prove that $x_0$ is a Lebesgue density point of $A_0$
if and only if $x_0'$ is a Lebesgue density point of $A_0$.

\medskip{}

The argument for unstable saturation is more complicated.
Assume now that $x$ and $x'$ are points in $M_f$ connected by a short unstable
segment and define $x_0 = \pi(x)$ and $x_0' = \pi(x').$
We may assume that $x,$ $x',$ and the unstable segment between them
all lie in $\pi \inv(U) = U \ti \pi \inv(x_0)$ for some small open disc $U.$
The problem now is that we are not free here to choose a ``good'' copy of $U.$
It may be the case that $U \ti x$ is the only copy of $U$ inside of
$\pi \inv(U) = U \ti \pi \inv(x_0)$ for which the two points in the fibers over $x_0$
and $x_0'$ lie on the same unstable manifold.
Further, $U \ti x$ may be a ``bad'' manifold in the sense that
the symmetric difference between
$A^{rs} \cap (U \ti x)$ and $A^{u} \cap (U \ti x)$ may have positive measure.
Therefore, we cannot simply apply the \cite{BW10} argument inside of $U \ti x.$

Instead, we define a new sequence of juliennes
$\Jrsc_n(x)$ which also extend into the fiber or ``$r$'' direction.
A julienne of this form contains infinitely many copies of $U,$
most of which are ``good'' in the sense of Birkhoff's Ergodic Theorem
and we can show that
\[
    \lim_{r \searrow 0} \mu(A_0 : B_r(x)) = 1
    \quad \Leftrightarrow \quad
    \lim_{n \to \infty} \widehat{\nu}^{rsc} (A^{rs} : \Jrsc_n(x)) = 1.
\]
for an appropriately defined measure $\widehat{\nu}^{rsc}$ on the julienne.

We adapt the internested holonomy argument in \cite{BW10} to show that
\[
    \lim_{n \to \infty} \widehat{\nu}^{rsc}(A^{rs} : \Jrsc_n(x)) = 1
    \quad \Leftrightarrow \quad
    \lim_{n \to \infty} \widehat{\nu}^{rsc}(A^{rs} : \Jrsc_n(x')) = 1
\]
and from this we conclude that $x_0$ is a Lebesgue density point for $A_0$
if and only if $x_0'$ is.
For the definition of essential accessibility in the non-invertible setting,
these forms of stable and unstable saturation of $A_0$ are enough to
conclude that the system is ergodic.

\medskip{}

The summary in the above paragraphs glosses over many subtleties arising in
the proofs. The full details are given in the sections that follow this
introduction. Many of the lemmas are reformulations of corresponding lemmas in
\cite{BW10}. If the proofs are similar enough to the original versions, we simply refer
the reader to the corresponding lemma in \cite{BW10}. In some cases, the proof differs
enough that we provide an adapted proof in Appendix \ref{Appendix proofs}. Some results are
unique to the non-invertible case and the proof appears in the main text.

\subsection{K-property}
Burns and Wilkinson establish the Kolmogorov or K-property
for the partially hyperbolic diffeomorphisms that they consider.
In fact, they prove ergodicity and then use a result of Brin and Pesin 
\cite{BP74}
to show that the system has the K-property.
In Appendix \ref{appendix:K-property}, we adapt these arguments to show that
the invertible system $\sig \colon M_f \to M_f$ on the space of orbits has the
K-property with respect to the lifted measure $\nu.$
This implies, among other ergodic properties,
that $\sig$ is mixing with respect to $\nu$
and therefore the original endomorphism $f \colon M \to M$
is mixing with respect to the invariant volume $\mu.$
\subsection*{Acknowledgements}
The authors would like to thank Jason Atnip, Keith Burns, Marisa Dos Reis Cantarino, Matilde Martínez, Rafael Potrie, Jana Rodriguez Hertz, Warwick Tucker, Raúl Ures and Amie Wilkinson for useful conversations.

\section{Topological structure of the space of orbits}
\label{Sec: Topological structure of the space of orbits}
This section looks at the topological structure of the space of orbits $M_f$
in more detail.
Let the partially hyperbolic endomorphism $f \colon M \to M,$
space of orbits $M_f,$
shift map $\sig \colon M_f \to M_f,$ and 
projection $\pi \colon M_f \to M$
be as in the introduction.

In order to define a metric on $M_f$, first
fix a small positive constant $0 < \lam^r < 1.$
In particular, we assume that
\[
    \lam^r < \| Df(v) \|
\]
holds for all unit vectors $v \in TM.$
Write $\lam^r_n$ for the $n$-th power of $\lam^r.$
That is, $\lam^r_n = [\lam^r]^n$.
This switches the usual roles of subscript and superscript,
but doing so is consistent with the cocycle notation
$\lam^s_n, \lam^c_n, \hat\lam^c_n, \lam^u_n$
used in later sections.
As mentioned in the introduction, ``$r$'' stands for fiber direction.

Since $f : M \to M$ is a covering map, there is a radius
$0 < \rcov \le 1$ with the property that for any $x \in M,$
the preimage $f \inv(B(x, \rcov))$ has exactly $\deg(f)$ connected components
and $\dist(U, V) > \rcov$ for any two distinct connected components of this
preimage.

By rescaling the Riemannian metric on $M,$ we may assume that $\rcov = 1$ and that the injectivity radius of $f$ on the base manifold is greater than 1.

Define a metric $\dcap$ on $M$ by
$\dcap(x, y) = \min \{ d_M(x, y), 1 \}$
where $d_M$ is the usual metric on $M$ coming from the Riemannian metric.
Then define a metric $d$ on the space of orbits by
\[
    d(x,y) = \sup_{n \ge 0} \{ \lam^r_n \dcap(x_{-n}, y_{-n}) \}
\]
for orbits $x = \{x_n\}$ and $y = \{y_n\}$ in $M_f.$
This metric has several nice properties.

To state the first of these properties, we define
\begin{align*}
    \mathcal{W}^{r}(x) & \coloneqq \{y\in M_f : {\exists n\in \mathbb{Z}}\allowbreak \text{ such that } {x_n=y_n}\},\qandq \\
    \mathcal{W}^{r}\loc(x) & \coloneqq \{y\in M_f : x_0=y_0\}.
\end{align*}
Note that ${\mathcal{W}^{r}\loc(x)=\pi^{-1}(x_0)}$.
\begin{lema} \label{lemma:fibermetric}
    For distinct orbits $x = \{x_n\}$ and $y = \{y_n\}$ with $y \in \Wrloc(x),$
    the metric satisfies
    $d(x, y) = \lam^r_m$ where $m \in \bbZ$ is such that
    $x_n = y_n$ for all $n > -m$ and
    $x_{-m} \ne y_{-m}.$
\end{lema}
\begin{proof}
    Since $y \in \Wrloc(x)$ such an integer $m > 0$ exists.
    The assumption above that $\rcov = 1$ implies that
    $\dcap(x_{-m}, y_{-m}) = 1.$
    Since $\dcap(x_n, y_n) = 0$ for all $n > -m$
    and $\dcap(x_n, y_n) \le 1$ for all $n < -m,$
    it follows that
    $d(x, y) = \lam^r_m.$
\end{proof}
To state the second property of the metric \( d \), we define \( \mathcal{W}^{scu}(x) \), the \textit{global scu-sheet through \( x \)}, as the connected component of \( M_f \) that contains \( x \).
We also define \( \mathcal{W}^{scu}(x, r) \) as the connected component of \( \mathcal{W}^{scu}(x) \cap B(x, r) \) that contains \( x \).
Here, \( B(x, r) \) denotes the ball of radius \( r \) centred at \( x \) with respect to the metric $d$.
\begin{lema} \label{lemma:scumetric}
    For orbits $x = \{x_n\}$ and $y = \{y_n\}$ with $y \in \Wscu(x, 1),$
    the metric satisfies 
    $d(x, y) = d_M(x_0, y_0).$
\end{lema}
\begin{proof}
    
    By the Hopf-Rinow theorem, there is a path $\al_0 : [0,1] \to M$
    from $x_0$ to $y_0$ such that
    $\operatorname{length}(\al_0) = d_M(x_0, y_0).$
    Lift $\al_0$ to a path $\al : [0,1] \to \mathcal{W}^{scu}(x, 1)$
    from $x$ to $y,$ and then define a sequence of paths $\al_n : [0,1] \to M$
    by $\al_n = \pi \circ \sig^{-n} \circ \al.$
    In particular, $\al_n$ is a path from $x_n$ to $y_n$ such that $f^n \circ \al_n = \al_0.$
    The derivative $Df$ contracts vectors by less than $\lam^r$ and so
    \[
        \lam^r_n \dcap(x_{-n}, y_{-n})
        \ \le \ \lam^r_n \operatorname{length}(\al_n)
        \ < \ \operatorname{length}(\al_0)
        \ = \ d_M(x_0, y_0).
    \]
    Taking the supremum over all $n \ge 0$ gives $d(x, y) = d_M(x_0, y_0).$
\end{proof}

In \cite{aoki1994topological}, it is shown that if $\deg(f)>1$ then $(M_f,M,C,\pi)$ is a fibre bundle where $C$ denotes the Cantor set.
In particular, the next result follows from \cite[Theorem 6.5.1]{aoki1994topological}. \par 
\begin{prop}
\label{canon}
    If $U\subset M$ is an open and simply connected set and $x_0\in U$, then there is a canonical homeomorphism from $ U\times \pi^{-1}(x_0)$ to $\pi^{-1}(U)$. Moreover, $\pi^{-1}(x_0)$ is homeomorphic to the Cantor set.
\end{prop}

From this point on, if $x_0 \in U \subset M$ are as in this proposition,
we will abuse notation and treat $\pi^{-1}(U)$ and $U\times \pi^{-1}(x_0)$ as
if they are equal.

The naming of the $scu$-sheets comes down to the fact that with the associated differential structure  $$T\mathcal{W}^{scu}(x) = \bigcup_{y \in \mathcal{W}^{scu}(x)} T_y M_f$$ and the dominated splitting in that tangent bundle for any point $y\in\mathcal{W}^{scu}(x)$, $$T_y \mathcal{W}^{scu}(x) = E^s_y \oplus E^c_y \oplus E^u_y$$ we have that the $scu$-sheet is tangent to the stable, center and unstable directions for any point in it.
By Proposition \ref{canon}, $\mathcal{W}^{r}\loc(x)$ is homeomorphic to the Cantor set.
Since $M$ is compact, it can be covered by finitely many open sets of the form $\pi^{-1}(B(x_0,1))$, and from Propositions \ref{canon} and \ref{lemma:scumetric}, 
$$\pi^{-1}(B(x_0,1))=\bigsqcup_{y\in \mathcal{W}^{r}\loc(x)} \mathcal{W}^{scu}(y,1).$$
This shows that the space of orbits is a laminated space with a Cantor set transverse structure.

The \( scu \)-sheets are subfoliated by the stable leaves \( \mathcal{W}^s \), tangent to \( E^s \). They are also subfoliated by unstable leaves \( \mathcal{W}^u \), tangent to \( E^u \). For $a=u,s$, a \textit{foliation box} for $\mathcal{W}^a$ is the image of $\R^{n-d}\times\R^n$ that sends each vertical $\R^d$-slice into a leaf of $\mathcal{W}^a$.

Let $R>0$ be 
small enough to
ensure that each \( \mathcal{W}^{scu}(x, R) \) is contained within foliation boxes for both \( \mathcal{W}^s \) and \( \mathcal{W}^u \). 
For this choice of \( R \), we define the \textit{local \( scu \)-sheet} as \( \mathcal{W}^{scu}\loc(x) \coloneqq \mathcal{W}^{scu}(x, R) \).

Similarly to \( \mathcal{W}^{scu}(x, r) \), for the stable and unstable foliations as well as the fake invariant foliations introduced in the next section, we will write \( \mathcal{W}^{\beta}(x, r) \) for the connected component of \( \mathcal{W}^{\beta}(x) \cap B(x, r) \) that contains \( x \).
We will also define \( \mathcal{W}^{\beta}\loc(x) \coloneqq \mathcal{W}^{\beta}(x, R) \).
In the $r$-direction, we define $\mathcal{W}^r(x,r)\coloneqq \mathcal{W}^r\loc(x) \cap B(x,r)$, complementing $\mathcal{W}^r(x)$ and $\mathcal{W}^r\loc(x)$.

\begin{lema}\label{lemma:fibermetricinvariance}
    Given $q\in \mathcal{W}^{scu}\loc(p)$,
    let $h^{scu}\colon \mathcal{W}^r\loc(p)\to\mathcal{W}^r\loc(q)$ be the $scu$-holonomy.
    Then $h^{scu}$ is an isometry.
\end{lema}
\begin{proof}
    To show that the distance $d$ is preserved by $scu$-holonomies we are going to define two curves on $\pi^{-1}(B(p_0,1) )$. The curves $\alpha$ and $\beta$ will be such that $\alpha(0),\beta(0)\in  \mathcal{W}^r\loc(p)$ and $\alpha(1),\beta(1)\in  \mathcal{W}^r\loc(q)$.
    Since $\alpha(0)$ and $\beta(0)$ can be any two points in $\mathcal{W}^r\loc(p)$, it is sufficient to show that $d(\alpha(0),\beta(0))=d(\alpha(1),\beta(1))$.

    Let $d(\alpha(0),\beta(0))=\lambda^r_m$.
    It can be shown that the curves $\alpha$ and $\beta$ can be chosen to also satisfy $\alpha_n(t)=\beta_n(t)$ for $t\in [0,1]$ and $n>-m$. Since we only care about the endpoints of $\alpha$ and $\beta$, assume that they satisfy that property.
    Then, by Lemma \ref{lemma:fibermetric}, for all $t\in[0,1]$ the distance $d(\alpha_n(t),\beta_n(t))$ is a power of $\lambda^r$.
    Because $\alpha_n(t)$ and $\beta_n(t)$ are continuous, $d(\alpha_n(t), \beta_n(t))$ is also continuous in $t$, leading to the result.
\end{proof}

Finally, it is straightforward to verify that if $y\in \mathcal{W}^r\loc(x)$, then $d(\sigma^n x,\sigma^n y) = \lambda_n^r\, d(x,y)$ and since $\lambda^r<\min_{v\in T^1M}\{\|Df(v)\|\} $ we have that the contraction in the $r$-direction is stronger than in the $s$-direction. 
\section{Fake invariant foliations}
We build fake invariant foliations in a similar fashion
to Section 3 of Burns and Wilkinson \cite{BW10}
which is itself based on Chapter 5 of \cite{HSP}.
In the diffeomorphism setting,
Burns and Wilkinson look at the orbit $\{f^n(p)\}$ through
a point $p \in M$ in order to build fake foliations defined in
a small ball $B(p, r) \subset M$ centered at $p.$
In the endomorphism setting,
we consider an orbit $p = \{p_n\}_{n \in \mathbb{Z}} \in M_f$ and
use this orbit to construct fake foliations
in a small ball $B(p_0, r) \subset M$ centered at $p_0.$
Using the projection $\pi : M_f \to M,$
we can then lift the fake foliations to a ball 
centered at a point $p \in M_f$ and
contained in the $scu$-leaf through that point.

To state the result, we recall the multiplicative cocycle notation
used in \cite{BW10}.
If $\alpha \colon M_f \to \R$ is a positive function, and $j\geq 1$ is an integer, let 
$$\alpha_j(p)\coloneqq \alpha(p)\alpha(\sigma p) \cdots \alpha(\sigma^{j-1}p ),\quad \text{and}\quad  \alpha_{-j}(p)\coloneqq \alpha(\sigma^{-j}p )^{-1}\alpha(\sigma^{-j+1}p )^{-1}\cdots\alpha(\sigma^{-1} p)^{-1}. $$
By convention, $\alpha_0(p)=1$. Observe that $\alpha_j$ is a multiplicative cocycle and that $(\alpha\beta)_j=\alpha_j \beta_j$.

\begin{prop} \label{prop:fakefoln}
    Let $f : M \to M$ be a $C^1$ partially hyperbolic endomorphism.
    For any $\varepsilon > 0,$
    there exist constants $r$ and $r_1$ with $R > r > r_1 > 0$
    such that for every orbit $p = \{p_n\}_{n \in \mathbb{Z}} \in M_f,$
    the neighbourhood $B(p_0, r) \subset M$ is foliated by foliations
    $\widehat{\mathcal{W}}^u_{p}$, $\widehat{\mathcal{W}}^{s}_{p}$, $\widehat{\mathcal{W}}^{c}_{p}$,
    $\widehat{\mathcal{W}}^{cu}_{p}$, and $\widehat{\mathcal{W}}^{sc}_{p}$
    with the following properties for each
    $\beta \in \{u, s, c, cu, sc\}.$
    \begin{enumerate}[label={\upshape(\roman*)}]
        \item \label{Lemma_i:Almost tangency to invariant distributions} \textnormal{Almost tangency to invariant distributions.}
        For each $q_0\in B(p_0,r)$, the leaf $\widehat{\mathcal{W}}^\beta_{p}(q_0)$ is $C^1$, and the tangent space $T_{q_0}\widehat{\mathcal{W}}^\beta_{p}({q_0})$ lies in a cone of radius $\varepsilon$ about $\pi_*E^\beta({q})$.
        \item \textnormal{Local invariance.}
        For each $q = \{q_n\}_{n \in \mathbb{Z}}\in \mathcal{W}^{scu} (p,r_1)$,
        $$ f(\widehat{\mathcal{W}}^\beta_p (q_0,r_1))\subset \widehat{\mathcal{W}}^\beta_{\sigma(p)}(f(q_0)) \quad 
        \text{and}\quad 
\widehat{\mathcal{W}}^\beta_p (q_0,r_1)\subset f(\widehat{\mathcal{W}}^\beta_{\sigma^{-1}(p)}(q_{-1})).$$
\item\label{Prop_iii: Exponential growth bounds at local scales} \textnormal{Exponential growth bounds at local scales.} Let $q = \{q_n\}_{n \in \mathbb{Z}} \in M_f$ and $q' = \{q'_n\}_{n \in \mathbb{Z}} \in M_f,$ then the following holds for all $n\geq 0$:
        \begin{enumerate}[label={\upshape(\alph*)}]
            \item \label{Prop_iii_a} Suppose that $q_j\in B(p_j,r_1)$ for $0\leq j\leq n-1$.\par
            If $q'_0\in \widehat{\mathcal{W}}^s_{p}(q_0,r_1)$, then 
            $$q'_n\in \widehat{\mathcal{W}}^s_{\sigma^np}(q_n,r_1)\quad \text{and} \quad d_M(q_n,q'_n)\leq \lambda^{s}_n(p) \cdot d_M(q_0,q'_0).$$\par
            If $q'_j\in \widehat{\mathcal{W}}^{sc}_{\sigma^jp}(q_j,r_1)$ for $0\leq j\leq n-1$, then 
            $$q'_n\in \widehat{\mathcal{W}}^{sc}_{\sigma^np} (q_n)\quad \text{and} \quad d_M(q_n,q'_n)\leq \hat{\lambda}^{c}_n(p)^{-1}\cdot d_M(q_0,q'_0).$$
            \item \label{Prop_iii_b} Suppose that $\sigma^{-j}q\in \mathcal{W}^{scu}(\sigma^{-j} p,r_1)$ for $0\leq j\leq n-1$.\par
            If $q'_0\in \widehat{\mathcal{W}}^u_{p}(q_0,r_1)$, then 
            $$q'_{-n}\in \widehat{\mathcal{W}}^u_{\sigma^{-n}p} (q_{-n},r_1)\quad \text{and} \quad d_M(q_{-n},q'_{-n})\leq \lambda^{u}_{-n}(p)^{-1} \cdot d_M(q_0,q'_0).$$\par
            If $q'_{-j} \in \widehat{\mathcal{W}}^{cu}_{\sigma^{-j}p}(q_{-j},r_1)$ for $0\leq j\leq n-1$, then 
            $$q'_{-n}\in \widehat{\mathcal{W}}^{cu}_{\sigma^{-n}p} (q_{-n})\quad \text{and} \quad d_M(q_{-n},q'_{-n})\leq {\lambda}^{c}_{-n}(p)\cdot d_M(q_0,q'_0).$$
        \end{enumerate}
        \item \textnormal{Coherence.} 
        $\widehat{\mathcal{W}}^s_p$ and $\widehat{\mathcal{W}}^c_p$ subfoliate $\widehat{\mathcal{W}}^{sc}_p$, and $\widehat{\mathcal{W}}^c_p$ and $\widehat{\mathcal{W}}^u_p$ subfoliate $\widehat{\mathcal{W}}^{cu}_p$.
        \item \label{Prop_v: Uniqueness} \textnormal{Uniqueness.}
        $\widehat{\mathcal{W}}^s_p(p_0)= \pi (\mathcal{W}^s(p,r))$,
        and $\widehat{\mathcal{W}}^u_p(p_0)= \pi (\mathcal{W}^u(p,r))$.
        \item \textnormal{Regularity.} If $f$ is $C^{1+\alpha}$, then the foliations $\widehat{\mathcal{W}}_p^{s}$, $\widehat{\mathcal{W}}_p^{c}$, $\widehat{\mathcal{W}}_p^{u}$, $\widehat{\mathcal{W}}_p^{sc}$, and $\widehat{\mathcal{W}}_p^{cu}$ and their tangent distributions are uniformly Hölder continuous.
        \item \label{Lemma_vii:Regularity_Strong_Foliation} \textnormal{Regularity of the strong foliation inside weak leaves.} If $f$ is $C^2$ and center bunched, then each leaf of $\widehat{\mathcal{W}}^{sc}_{p}$ is $C^1$ foliated by leaves of the foliation $\widehat{\mathcal{W}}^{s}_{p}$, and each leaf of $\widehat{\mathcal{W}}^{cu}_{p}$ is $C^1$ foliated by leaves of the foliation $\widehat{\mathcal{W}}^{u}_{p}$. If $f$ is $C^{1+\alpha}$ and strongly center bunched, then the same conclusion holds.
        \item \textnormal{Projection.} \label{Lemma_viii:Endos} If $p, q \in M_f$ are such that $\pi(p) = \pi(q),$
            then $\widehat{\mathcal{W}}^{s}_{p}(p_0) = \widehat{\mathcal{W}}^{s}_{q}(q_0)$

            and $\widehat{\mathcal{W}}^{sc}_{p}(p_0) = \widehat{\mathcal{W}}^{sc}_{q}(q_0).$
    \end{enumerate}
\end{prop}
Properties \ref{Lemma_i:Almost tangency to invariant distributions} to \ref{Lemma_vii:Regularity_Strong_Foliation} are analogous to those of the fake foliations built by Burns and Wilkinson in the diffeomorphism case.
Property \ref{Lemma_viii:Endos}
is only relevant in the endomorphism setting.
\begin{proof}[Proof of Proposition \ref{prop:fakefoln}]
    Similar to the construction of $T M_f$,
    we construct a vector bundle $E$ over $M_f$
    by setting each fiber $E_p$ for $p \in M_f$ equal to
    the tangent space $T_{\pi(p)} M.$
    For this proof, however, we equip $M_f$ with the discrete topology,
    so $E$ is a disjoint union of fibers.
    Note that $E$ has multiple copies
    of each tangent space $T_{p_0} M,$ one copy for every orbit
    passing through $p_0.$

    Choose a constant $r_0 > 0$ and define a map
    $F_r : TM \to TM$
    for $r \in (0,r_0]$
    exactly as in Step 1 of the proof of
    \cite[Prop 3.1]{BW10}.
    Note that the definition of $F_r$ there does not
    depend on $f$ being invertible.
    If $f$ is locally a $C^k$ diffeomorphism, this construction
    will produce a map $F_r$ such that each
    $F_r(p_0, \cdot) \colon T_{p_0} M \to T_{f(p_0)} M$
    is a $C^k$ diffeomorphism.
    Then define a map $\hat F_r \colon E \to E$ by
    $\hat F_r(p, v) = F_r(\pi(p), v)$
    where here we use that $E_p$ is identified with $T_{\pi(p)}M.$
    Using $\hat F_r$ in place of $F_r$ in the remainder
    of the proof in \cite[Prop 3.1]{BW10} yields the desired results.
    In particular, all of the assertions of uniformity of estimates
    for $F_r$ given there hold for $\hat F_r$ as well.\par 
Property \ref{Lemma_viii:Endos} follows because the graph transform argument used to construct
the stable and center-stable fake foliations only relies on the forward orbit
and here $p_n$ and $q_n$ agree for all $n \ge 0.$   
\end{proof}

It will be convenient in later sections to assume that $r \gg 1$
in Proposition \ref{prop:fakefoln} and similar to the proof in \cite{BW10},
so we assume that the rescaling of the
Riemannian metric done at the start of Section \ref{Sec: Topological structure of the space of orbits}
was also chosen to ensure that $r$ is much greater than one.
Note that we still assume that $\rcov = 1$ holds.

\begin{notation}
    As noted in Section \ref{Sec: Topological structure of the space of orbits},
    for any point $p \in M_f$,
    the projection
    $\pi$ restricts to a homeomorphism on the ball
    $\mathcal{W}^{scu}(p, R)$ inside of the $scu$-sheet through $p$.
    Therefore, if $x \in \mathcal{W}^{scu}(p, r)$,
    we write $\widehat{\mathcal{W}}^{\alpha}_p(x, r)$ to denote
    $\pi^{-1} (\widehat{\mathcal{W}}^{\alpha}_p( \pi(x), r)$.
    Later on, we will restrict our focus to one specific orbit $p$
    and therefore drop the subscript and just write $\widehat{\mathcal{W}}^{\alpha}(x, r)$.
\end{notation}
\section{Measures in the space of orbits}

This section defines several measures on the space of orbits. 
However, we first introduce definitions and results that are general in nature and apply to any measure.

Let $m$ be a measure on a space $X$.
Recall from the introduction that for $m$-measurable sets $A$ and $B$, with $m(B)>0$ we define the \textit{density of $A$ in $B$} as $$m(A:B):=\frac{m(A\cap B)}{m(B)}.$$ \par 
A sequence $Y_n$ of measurable sets is \textit{regular} if there exists $C>0$ and $K\geq 1$ such that for all $n\geq 0$ we have $m(Y_{n+k})\geq C m(Y_n)$. \par 
Two sequences of sets $\{Y_n\}$ and $\{Z_n\}$ are \textit{internested} if there exists a $k\geq 1$ such that for all $n\geq k$ we have $Y_{n+k}\subset Z_{n}\subset Y_{n-k}$.\par 

\begin{lema}\label{comparablesequences}
Let $Y_n$ and $Z_n$ be internested sequences of measurable sets, with $Y_n$ regular. Then $Z_n$ is also regular. If the sets $Y_n$ have positive measure, then so do the $Z_n $, and, for any measurable $A$,
$$\lim_{n\to \infty} m (A:Y_n)  =1 \iff  \lim_{n\to \infty} m (A: Z_n)  =1. $$
\end{lema}
The lemma is a consequence of the definitions of internested sequences and regular sequences.

Recall that throughout this paper that $\mu$ represents the measure associated
to a volume form on $M$ with constant Jacobian.
A point $x\in M$ is a \textit{Lebesgue density point} of a measurable set $A$ if
$$\lim_{r\to 0}\mu(A:B(x,r))=1.$$\par
\begin{theo}[Lebesgue Density Theorem]
\label{Lebesgue Density Theorem}
    If $A$ is a Lebesgue measurable set, then almost every $p\in A$ is a Lebesgue density point of $A$.
\end{theo}
A proof of this result can be found in Chapter 6 of \cite{pugh2013real}.

\medskip

We now define measures on the space of orbits, the $scu$-sheets, the leaves of both the true and fake foliations, and the $\mathcal{W}^r$ sets.

\begin{prop}[\textit{scu}-measure]
\label{Prop:scu measure}
    There is a Borel measure $\nu^{scu}$ defined on the sheets of the lamination $\mathcal{W}^{scu}$ such that, 
    for all $x\in M_f$ if $A\subset  \mathcal{W}^{scu}\loc (x)$ is a $\nu^{scu}$-measurable set then $\nu^{scu}(A)=\mu(\pi(A))$,
    and for all $\nu^{scu}$-measurable $A$ we have 
\begin{equation}
\label{eq:scu-expansion}
  \nu^{scu}(\sigma(A))=\deg(f)\,\nu^{scu}(A).  
\end{equation}
\end{prop}
\begin{proof}
    Since $\pi|_{\mathcal{W}^{scu}(x)}$ is a local diffeomorphism, we can pull back the volume form on $M$ to $\mathcal{W}^{scu}(x)$ using $\pi$. This defines the volume form on each $\mathcal{W}^{scu}(x)$.
    Equation \ref{eq:scu-expansion} follows from $f$ having constant Jacobian equal to $\deg{f} $.
\end{proof}
The following proposition states the properties of the measure \(\nu^r\) on \(M_f\).
\begin{prop}[\textit{r}-measure]
\label{Prop:r-measure}
There is a unique measure \(\nu^r\) on \(M_f\) satisfying:
\begin{enumerate}
    \item \(\nu^r(\mathcal{W}^r\loc(x)) = 1\) for all \(x \in M_f\), and
    \item \(\displaystyle\nu^r(\sig(A)) = \frac{1}{\deg(f)} \nu^r(A)\) for all \(\nu^r\)-measurable \(A \subset M_f\).
\end{enumerate}
\end{prop}

A full proof is provided in Appendix \ref{Appendix proofs}.
The idea of the proof is to observe that
\[
    \nu^r( \sig^n(\Wr_{loc}(x))) = [\deg(f)]^{-n}
\]
must hold for all points $x \in M_f$ and integers $n \in \bbZ.$
Using sets of this form to define a premeasure and applying the Carathéodory extension theorem,
we construct the measure.

Recall from Proposition \ref{canon} that if $x_0 \in U \subset M$ where $U$ is open and simply connected,
then $\pi^{-1}(U)$ is homeomorphic to $U \times \pi^{-1}(x_0)$
and we regard the two sets as equal.

\begin{prop}\label{defnu}
There exists a unique Borel probability measure $\nu$ on $M_f$ satisfying:
\begin{enumerate}
    \item $\nu$ is $\sigma$-invariant,
    \item $\mu$ is the pushforward measure $\pi_*\nu$, and
    \item
    if $x_0 \in U \subset M$ where $U$ is open and simply connected,
    then the restriction of $\nu$ to $\pi^{-1}(U)$ is equal
    to the restriction of $\nu^{scu} \times \nu^{r}$ to $U \times \pi^{-1}(x_0)$.
\end{enumerate}
\end{prop}
\begin{proof}
    Any Borel set $A \subset M_f$ can written as a finite disjoint union
    $A = A_1 \sqcup \cdots \sqcup A_n$
    where each $A_i$ is a subset of
    $\pi^{-1}(U_i) = U_i \times \pi^{-1}(x_i)$
    for some $x_i \in U_i \subset M$ with $U_i$ open and simply connected.
    Define $\nu$ by
    \[
        \nu(A) = \sum (\nu^{scu} \times \nu^r) (A_i).
    \]
    Using Propositions \ref{Prop:scu measure} and \ref{Prop:r-measure},
    one can verify that $\nu$ is a well-defined measure on $M_f$
    which satisfies all three listed properties.
    It is a standard result of the study of endomorphisms that
    any measure which satisfies the first two of these properties
    is unique. See \cite[Proposition I.3.1]{qian2009smooth} for details.
\end{proof}

Consider the fake stable-center foliation \( \widehat{\mathcal{W}}^{sc}_p \).
Each of the fake stable-center leaves is an immersed submanifold within the sheet \( \mathcal{W}^{scu}(p) \). 
Then, each leaf has a well-defined volume induced by its Riemannian metric.
If a set \( A \) is contained within a single leaf of \( \widehat{\mathcal{W}}^{sc}_p \) and is measurable on that leaf, we denote by \( \widehat{\nu}^{sc}(A) \) the induced Riemannian volume of $A$ in that leaf. 
We call \( \widehat{\nu}^{sc}\) the \textit{fake stable-center measure}.

For $x \in \mathcal{W}^{scu}(p,r)$, consider the set $\widehat{\mathcal{W}}^{rsc}_p(x) \coloneqq \mathcal{W}^r\loc(\widehat{\mathcal{W}}^{sc}_p(x))$. 
Each point $q$ in $\widehat{\mathcal{W}}^{rsc}_p(x)$ is identified uniquely by the pair $(y,z)\in\mathcal{W}^r\loc(x)  \times \widehat{\mathcal{W}}^{sc}_p(x)$ where $y=\mathcal{W}^r\loc(x) \cap \mathcal{W}^{scu}\loc(q) $ and $z= \mathcal{W}^r\loc(q) \cap \widehat{\mathcal{W}}^{sc}_p(x)$. 
Thus, $\widehat{\mathcal{W}}^{rsc}_p(x)$ can be identified as the product $\mathcal{W}^r\loc(x)\times \widehat{\mathcal{W}}^{sc}_p(x)$.
Using this identification, we can define $\widehat{\nu}^{rsc}\coloneqq \nu^r\times \widehat{\nu}^{sc}$ on $\widehat{\mathcal{W}}^{rsc}_p(x)$.

We now consider transverse absolute continuity and
we repeat the definition given in \cite{BW10}.
For a foliation $\mathcal{F}$, the \textit{holonomy map} $h_\mathcal{F} \colon \tau_1 \to \tau_2$ is defined between two complete transversals $\tau_1$ and $\tau_2$ in a foliation box, mapping a point in $\tau_1$ to the intersection of its local leaf of $\mathcal{F}$ with $\tau_2$.
A foliation $\mathcal{F}$ with smooth leaves is \textit{transversely absolutely continuous with bounded Jacobians} if, for every $\alpha\in(0,\pi)$, there exists $C\geq 1$ such that, for every foliation box $U$ of diameter less than $R$, any complete transversals $\tau_1,\tau_2$ to $\mathcal{F}$ in $U$ of angle at least $\alpha$ with $\mathcal{F}$, and any $m_{\tau_1}$-measurable set $A$ contained in ${\tau_1}$, we have the inequality:
$$\frac{1}{C} m_{\tau_1}(A)\leq m_{\tau_2}(h_\mathcal{F}(A))\leq Cm_{\tau_1}(A). $$
The stable and unstable foliations of a partially hyperbolic diffeomorphism are transversely absolutely continuous with bounded Jacobians \cite[\S 3]{BP74}. 
This result can be adapted to our case to get the following result. 

\begin{prop}
    The unstable foliation inside of an $scu$-sheet
    is transversely absolutely continuous with bounded Jacobians.
    Moreover, the constant $C$ used in the definition
    is uniform over $M_f$ and independent of the choice
    of $scu$-sheet.
\end{prop}

In particular, we are interested in the case where the transversals are fake stable-center manifolds. 
For a point $q\in M_f$, let $x,x'\in \mathcal{W}^u\loc(q)$ be such that
$\widehat{\mathcal{W}}^{sc}_q(x)$ and $\widehat{\mathcal{W}}^{sc}_q(x')$ are
defined.
Then for any $\widehat{\nu}^{sc}$-measurable set $A$ contained in $\widehat{\mathcal{W}}^{sc}_q(x)$, we have the inequality:
\[
    \frac{1}{C}\; \widehat{\nu}^{sc}(A)\leq \widehat{\nu}^{sc}
    (h^u_q(A))\leq C\; \widehat{\nu}^{sc} (A)
\]
where $h^u_q\colon \widehat{\mathcal{W}}^{sc}_q(x)\to\widehat{\mathcal{W}}^{sc}_q(x')$ is the unstable holonomy.
\section{Proving the main theorem}
\label{Sec: The main theorem}
In this section, we state several results which combine together to
give the proof of Theorem \ref{main}.
The proofs of these results are then given in the remaining sections of this
paper.


Recall that there is a stable foliation $\mathcal{W}^{s}$ on $M$ that is tangent to the stable bundle $E^{s} \subset TM$.
Although we do not have an unstable foliation in $M$, the space of orbits \( M_f \) admits an unstable foliation \( \mathcal{W}^{u} \), tangent to the unstable bundle \( E^{u} \subset TM_f \).
For any $x \in M_f$ in the space of orbits,
let $\Wu(x)$ denote entire unstable leaf through $x.$
It is an injectively immersed copy of $\bbR^n$ into $\mathcal{W}^{scu}(x)$
where $n = \dim(E^u).$
Projecting down to $M,$
the image of $\pi(\Wu(x))$ is a copy of $\bbR^n,$
now immersed in the manifold $M$ and passing through $x_0 = \pi(x),$
but this immersion in $M$ is not necessarily injective
and $\pi(\Wu(x))$ may pass through itself.
We call $\pi(\Wu(x))$ \emph{an entire unstable manifold} passing through $x_0.$
Note that there are multiple such manifolds,
one for every orbit passing through $x_0.$

In the base manifold $M$,
a subset \(A \subset M \) is \textit{\( s \)-saturated} if it is the union of entire stable leaves in \( M \).  
It is \textit{\( u \)-saturated} if, for every \( x_0 \in A \), it contains every entire unstable manifold through \( x_0 \) in \( M \).  
A subset of \( M \) is said to be \textit{bisaturated} if it is both \( s \)-saturated and \( u \)-saturated. 
In these terms accessibility is equivalent to: if a set is bisaturated then it must be empty or all of $M$.
Further, essential accessibility is equivalent to: if a measurable set is bisaturated then it must have either zero or full measure.

In the space of orbits \( M_f \), the definitions are analogous. A subset of \( M_f \) is called:
\textit{\( u \)-saturated} if it is the union of entire unstable leaves in \( M_f \),
\textit{\( s \)-saturated} if it is the union of entire stable leaves in \( M_f \), and
\textit{\( r \)-saturated} if it is the union of entire \( \mathcal{W}^r(x) \) sets in \( M_f \).  
A subset of \( M_f \) is \textit{\( rs \)-saturated} if it is both \( r \)-saturated and \( s \)-saturated.
For \( \alpha \in \{u, rs\} \), a $\nu$-measurable subset \( X \subset M_f \) is
\textit{essentially \( \alpha \)-saturated} if there exists
an \( \alpha \)-saturated set \( X^\alpha \) such that
\( \nu(X \triangle X^\alpha) = 0 \). We call $X^\alpha$ the \textit{$\alpha$-saturate} of $X$.

\begin{theo}[Bi-essential theorem]
\label{bi}
Let $f\colon M\to M$ be $C^2$, partially hyperbolic and center bunched.
Let $A^{rs}\subset M_f$ be a measurable set that is both $rs$-saturated and essentially $u$-saturated. 
Then the set of Lebesgue density points of $\pi(A^{rs})$ is bisaturated.
\end{theo}
The proof of Theorem \ref{bi} splits into two parts: Lemma \ref{lemma:A0s-sat}, which establishes $s$-saturation, and Lemma \ref{lemma:A0u-sat}, which establishes $u$-saturation. We prove these lemmas in Sections \ref{Sec: Stable saturation} and \ref{Sec: Unstable saturation}, respectively.\par
For a set $A^{rs}$ as in Theorem \ref{bi}, define $A_0\coloneqq \pi(A^{rs})$.
\begin{lema}\label{lemma:A0s-sat}
    The set of Lebesgue density points of $A_0$ is $s$-saturated.
\end{lema}
\begin{lema}\label{lemma:A0u-sat}
    The set of Lebesgue density points of $A_0$ is $u$-saturated.
\end{lema}
We now show that Theorem \ref{bi} implies Theorem \ref{main} by a version of the Hopf argument.
\begin{proof}[Proof of Theorem \ref{main}] 
    Let $\varphi_0:M\to \R$ be a continuous function, and let $\varphi\coloneqq \varphi_0 \circ \pi : M_f \to \bbR$
    be the lift of $\varphi_0$ to the space of orbits.
    Let $\varphi^+$ be the forward Birkhoff average of $\varphi$ under the shift map $\sigma$.
    Since $\varphi$ is uniformly continuous
    and $\sigma$ contracts distances in both the $r$ and $s$ directions,
    the function $\varphi^+$ is constant along both $\Wr$ sets and stable leaves.
    Hence, for any $a\in\R$, the set $A^{rs}:=\{x\in M_f:\varphi^+(x)\leq a  \}$ is $rs$-saturated.
    Since the backward Birkhoff average $\varphi^-$ is constant on unstable leaves,
    and $\varphi^+=\varphi^-$ holds almost everywhere,
    the set $A^{rs}$ is essentially $u$-saturated.

    From Theorem \ref{bi},
    the set $X \subset M$ of Lebesgue density points of $A_0 = \pi(A^{rs})$ is bisaturated in $M$.
    Essential accessibility then implies that $X$ has zero or full measure.
    By the Lebesgue density theorem, $A_0$ itself has zero or full measure.
    Note that $A_0=\{x\in M \colon \varphi^+_0(x)\leq a\}$ where $\varphi^+_0$ is the forward Birkhoff average of $\varphi_0$.
    Since $a \in \bbR$ was arbitrary, this shows that $\varphi^+_0$ is
    essentially constant and so $f$ is ergodic. 
\end{proof}

\section{Stable saturation}\label{Sec: Stable saturation}
We now prove that the set of Lebesgue density points of $A_0$ is saturated by
stable leaves.
Recall that $A_0$ is defined as $\pi(A^{rs})$ and $A^{rs}\subset M_f$ is a measurable set that is both $rs$-saturated and essentially $u$-saturated. 
Let $A^u$ be the $u$-saturate of $A^{rs}$ that is, $A^u$ is $u$-saturated and $ \nu(A^u \triangle A^{rs}) = 0$.

Given a measurable set $X\subset \mathcal{W}^{scu}(x)$ inside an $scu$-sheet,
a point $y$ is an \textit{$scu$-Lebesgue density point} of $X$ if
$$\lim_{r\to 0}{\nu^{scu}(X: \mathcal{W}^{scu}(y,r))}=1.$$
\begin{lema}
\label{Lemma:scu_density_equivalence}
    A point $x\in M_f$ is an $scu$-Lebesgue density point of $A^{rs}$ if and only if $\pi(x)$ is a Lebesgue density point of $A_0$.
\end{lema}
\begin{proof}
    For a point \( x \in M_f \) and radius $r > 0$,
    we have by Proposition \ref{Prop:scu measure} that
    $\nu^{scu}(\mathcal{W}^{scu}(x, r)) = \mu(B(x_0, r))$ and
    $\nu^{scu}(A^{rs} \cap \mathcal{W}^{scu}(x, r)) = \mu(A_0 \cap B(x_0, r)).$
    The lemma follows from this.
\end{proof}

A \textit{good sheet} is an $scu$-sheet \( \mathcal{L} \) such that
the sets  $A^u$  and  $A^{rs}$  agree almost everywhere inside of $\mathcal{L} $ with respect to the  $\nu^{scu}$ measure, that is,
$$\nu^{scu}(\, (A^u \cap \mathcal{L})\,  \triangle \, (A^{rs} \cap \mathcal{L})\, ) = 0.$$
In particular, $A^{rs}$ is essentially $u$-saturated inside of a good sheet.

\begin{lema}
    \label{Lemma:Birkhoff_typical_sheets}
    For any point $x_0\in M$ and for $\nu^r$-a.e. $x\in \pi^{-1}(x_0)$, the sheet \( \mathcal{W}^{scu}(x) \) is a good sheet.
\end{lema}
\begin{proof}
    From Proposition \ref{defnu}, the restriction of \( \nu \) to \( \pi^{-1}(B(\pi(x), R)) \) is \( \nu^{scu} \times \nu^{r} \).
    Since $ \nu(A^u \triangle A^{rs}) = 0$ it follows by Fubini's theorem that
    for every \( x_0 \in M \) and \( \nu^r \)-almost every \( x \in  \pi^{-1}(x_0)\), that
    $$\nu^{scu}((A^u\cap\mathcal{W}^{scu}\loc(x)) \triangle (A^{rs}\cap\mathcal{W}^{scu}\loc(x))) = 0.$$
    Then as each global $scu$-sheet is the countable union of local $scu$-sheets,
    the result follows.
\end{proof}
\begin{lema}
\label{Lemma:SCU_density_saturated}
Let $\mathcal{L}$ be a good sheet,
then the set of $scu$-Lebesgue density points of $A^{rs}$ is $s$-saturated inside of $\mathcal{L}$.
\end{lema}
\begin{proof}
The proof of \cite[Theorem 5.1]{BW10} assumes that the dynamics are a $C^2$, partially hyperbolic and center bunched diffeomorphism on a compact manifold.
Under these hypotheses, they show that for any set $A$ that is both essentially $s$-saturated and essentially $u$-saturated, the set of Lebesgue density points of $A$ is bisaturated.

We can recreate the conditions necessary for the proof, since the restriction of $\sigma$ to the orbit of $\mathcal{L}$ is a $C^2$, partially hyperbolic and center bunched diffeomorphism.
Although $\mathcal{L}$ is not compact, the space of orbits $M_f$ is compact. 
Thus, uniform estimates, uniform angles, and other arguments that rely on continuity in a compact setting still apply by considering that $\sigma$ operates in the compact space $M_f$.
Finally, Proposition \ref{prop:fakefoln} states uniformity results identical to the fake invariant foliations used by Burns and Wilkinson.
All of this together allow us to recreate the proof of \cite[Theorem 5.1]{BW10} inside the $scu$-sheet.
This shows that the set of $scu$-Lebesgue density points of $A^{rs}$ is bisaturated.
\end{proof}

\begin{proof}[Proof of Lemma \ref{lemma:A0s-sat}]
    Let $x_0 \in M$ be a Lebesgue density point of $A_0$, and
    consider a point $y_0 \in \mathcal{W}^s(x_0, R)$ on its local stable
    manifold.
    It suffices to prove that $y_0$ is also a Lebesgue density point of $A_0$. 
    By Lemma \ref{Lemma:Birkhoff_typical_sheets}, there exists $x\in \pi^{-1}(x_0)$ such that \( \mathcal{W}^{scu}(x) \) is a good sheet.
    Let $y\in M_f$ be the unique point in the intersection $\mathcal{W}^s\loc(x)\cap \pi^{-1}(y_0)$. 
    By Lemma \ref{Lemma:scu_density_equivalence}, $x$ is an $scu$-Lebesgue density point.
    By Lemma \ref{Lemma:SCU_density_saturated} , we see that $y$ is also an $scu$-Lebesgue density point.
    Then, applying Lemma \ref{Lemma:scu_density_equivalence} again, $y_0$ is a Lebesgue density point of $A_0$.
    Therefore, the set of Lebesgue density points of $A_0$ is $s$-saturated.
\end{proof}

It is tempting to apply the same argument to establish \( u \)-saturation. However, the reason this approach works for \( s \)-saturation but not for \( u \)-saturation lies in the properties of the stable and unstable foliations. The stable foliation on each \( scu \)-sheet projects down to the stable foliation on \( M \), allowing any stable curve on \( M \) to be lifted to an \( scu \)-sheet. Moreover, this lift is guaranteed to be tangent to the stable distribution on that \( scu \)-sheet.

In contrast, for the unstable saturation case, an unstable curve on \( M \) may only lift to an \( scu \)-sheet where \( A^{rs} \) is not essentially \( u \)-saturated. This limitation prevents us from satisfying the conditions necessary to apply Lemma \ref{Lemma:SCU_density_saturated}, making it impossible to use the same argument for \( u \)-saturation.
Figures \ref{fig:StableProjection} and \ref{fig:multipleunstables} illustrate the differences between the stable foliation and the unstable foliation on each $scu$-sheet, respectively.
\section{Unstable saturation}
\label{Sec: Unstable saturation}
We will establish the $u$-saturation of the set of Lebesgue density points of $A_0$,
this result is necessary for the proof of Theorem \ref{bi}.
Based on Burns and Wilkinson's arguments \cite[Section 5]{BW10}, we will show the following.
For any \( p_0 \in M \) and \( p \in \pi^{-1}(p_0) \), if \( x_0, x'_0 \in \pi(\mathcal{W}^u(p,1)) \) and \( x_0 \) is a Lebesgue density point of \( A_0 \), then \( x'_0 \) is also a Lebesgue density point.
We restate these arguments to establish definitions and notation in order to maintain consistency.
Additionally, it prepares for the discussions in Sections \ref{SEC:Internested property of the $rsc$-juliennes} and \ref{Section: Density of rsc-julienne}, which complete the proof of Lemma \ref{lemma:A0u-sat}.

Recall that since $f$ is center bunched we have positive functions $M_f \to \bbR$ which satisfy
\begin{equation*}
\lambda^s<\lambda^c\cdot\hat\lambda^c\quad\text{and}\quad\lambda^u<\lambda^c\cdot\hat\lambda^c.
\end{equation*}
Burns and Wilkinson use the first inequality, while we will use the second. This is because we show $u$-saturation instead of $s$-saturation.
By the second inequality, we can choose functions $\tau,\delta\colon M_f\to\R$ such that
\begin{equation}
\label{eq:def tau and delta}
    \lambda^u<\tau<\delta\cdot\hat \lambda^c\quad \text{and}\quad \delta<\min\{\lambda^c,1\}.
\end{equation}

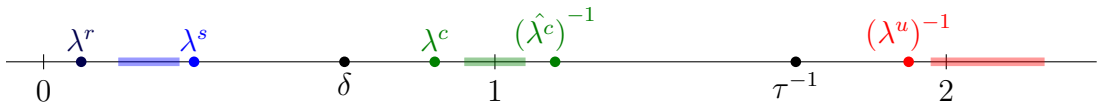
\begin{figure}[ht]
    \centering
  \begin{tikzpicture}
    \draw (-0.5,0) -- (14,0);
\foreach \i in {0,...,2} 
      \draw (6*\i,0.1) -- + (0,-0.2) node[below] {$\i$}; 
\fill[blue!30!black] (.5,0) circle (2pt) node[above] {$\lambda^r$};
\fill[blue] (2,0) circle (2pt) node[above] {${\lambda^s}$};
\fill[green!50!black] (5.2,0) circle (2pt) node[above] {$\lambda^c$};
\fill[green!50!black] (6.8,0) circle (2pt) node[above] {${(\hat{\lambda^c})}^{-1}$};
\fill[red] (11.5,0) circle (2pt) node[above] {${(\lambda^u)}^{-1}$};
\filldraw [fill=blue, draw=blue, opacity=0.4] (1,-0.05) rectangle ++(0.8,0.1); 
\filldraw [fill=red, draw=red, opacity=0.4] (11.8,-0.05) rectangle ++(1.5,0.1);
\filldraw [fill=green!50!black, draw=green!50!black, opacity=0.4] (5.6,-0.05) rectangle ++(0.8,0.1); 
\fill[black] (10,0) circle (2pt) node[below] {$\tau^{-1} $};
\fill[black] (4,0) circle (2pt) node[below] {$\delta$};
  \end{tikzpicture}
\caption{Comparison of \(\lambda^r, \lambda^s, \delta, \lambda^c, \hat{\lambda}^c, \tau\), and \(\lambda^u\), illustrating how the center bunching condition ensures that \(\tau\) and \(\delta\) can be chosen to maintain the necessary inequalities in (\ref{eq:def tau and delta}).}
\end{figure}
Define $N^{scu} \coloneqq \bigsqcup_{j \geq 0} \mathcal{W}^{scu}({\sigma^{-j}p}, r)$, as the disjoint union of the balls $\mathcal{W}^{scu}({\sigma^{-j}p}, r)$ for $j \geq 0$, with $r$ given by Proposition \ref{prop:fakefoln}.
Recall that we rescaled the metric to assume that $r \gg 1$.
To simplify the notation, we omit the dependence on $p$ for the definition of the fake invariant foliations in $N^{scu}$.
Let $\widehat{\mathcal{W}}^u$ be the locally invariant foliation of $N^{scu}$ satisfying the following condition.
Similarly, we define foliations $\widehat{\mathcal{W}}^{s}$, $\widehat{\mathcal{W}}^{c}$, $\widehat{\mathcal{W}}^{cu}$, and $\widehat{\mathcal{W}}^{sc}$.

Burns and Wilkinson's argument for $s$-saturation takes place inside a set like $N^{scu}$.
However, because we are working on the space of orbits, we must account for the $r$-direction.
Define 
$$N\coloneqq \bigsqcup_{j \geq 0} \mathcal{W}^r\loc(\mathcal{W}^{scu}({\sigma^{-j}p}, r)).$$
Note that $N^{scu} \subset N$, so we still have the fake invariant foliations defined on $N^{scu}$.
Everything we do in the rest of this paper takes place within $N$.

For the remainder of the paper we will evaluate cocycles at the point \(p\)
and so we drop the dependence on \(p\) from the notation;
thus, if \(\alpha\) is a cocycle, then \(\alpha_n(p)\) will be abbreviated to \(\alpha_n\).

Following Burns and Wilkinson, we define the following sequences of sets.
For each $x \in \mathcal{W}^u(p,1)$, define
$$\widehat{B}^c_n(x)\coloneqq \widehat{\mathcal{W}}^c(x,\delta_n). $$
For $y\in \bigcup_{x\in \mathcal{W}^u(p,1)} \widehat{B}^c_n(x)$, define
$$\widehat{J}^s_n(y) \coloneqq \sigma^n(\widehat{\mathcal{W}}^s(\sigma^{-n}y,\tau_n))
\quad \text{and}\quad
J^s_n(y)\coloneqq \sigma^n(\mathcal{W}^s(\sigma^{-n}y,\tau_n)).$$
To define the $r$-juliennes, first define
\(
    \mathcal{W}^r(y, \tau_n) =
    \{ z \in \mathcal{W}^r\loc(y) : d(y,z) \le \tau_n \}.
\)
Then, $J^r_n(y)\coloneqq \sigma^n(\mathcal{W}^r(\sigma^{-n}y,\tau_n))$.
Since $d(\sig(y), \sig(z)) = \lambda^r d(y, z)$ for $y \in \mathcal{W}^r\loc(z)$,
this julienne can also be written as $J^r_n(y) = \mathcal{W}^r(y,\lam^r_n\tau_n)$.
The \textit{$sc$-julienne} and the \textit{$rsc$-julienne} of order $n$ centred at $x$
are defined as
\[
    \widehat{J}^{sc}_{n}(x)\coloneqq
    \bigcup_{y\in \widehat{B}^c_n(x)} \widehat{J}^s_n(y)
    \quad\text{and}\quad
    \widehat{J}^{rsc}_{n}(x) \coloneqq 
    \bigcup_{y\in \widehat{J}^{sc}_{n}(x)}J^r_n(y).
\]
The sets $\widehat{J}^{sc}_{n}(x)$ are defined analogously to the center-unstable juliennes introduced by Burns and Wilkinson. 
In our setting, we must also consider the $r$-direction, and hence we need to define the $rsc$-julienne.

The sets \(\widehat{J}^{rsc}_{n}(x)\) are contained in \(\widehat{\mathcal{W}}^{rsc}\loc(x)=\mathcal{W}^r\loc(\widehat{\mathcal{W}}^{sc}(x))\), where we have defined the measure \(\widehat{\nu}^{rsc}\).
A point $x \in M_f$ is an \textit{$rsc$-julienne density point} of a measurable set $X$ if
$$\lim_{n\to \infty}\widehat{\nu}^{rsc} (X :\widehat{J}^{rsc}_{n}(x) )=1, \quad 
\text{where} \quad \widehat{\nu}^{rsc}(X: \widehat{J}^{rsc}_{n}(x)  )= 
\frac{\widehat{\nu}^{rsc} (X \cap \widehat{J}^{rsc}_{n}(x)  )}{\widehat{\nu}^{rsc} ( \widehat{J}^{rsc}_{n}(x)  )}.$$

As defined in Section \ref{Sec: The main theorem},
$A_0 = \pi(A^{rs})$ where $A^{rs} \subset M_f$ is a measurable set that is both $rs$-saturated and essentially $u$-saturated. 
Further, $A^u\subset M_f$ is the {$u$-saturate} of $A^{rs}$, that is, $A^u$ is $u$-saturated and $ \nu(A^u \triangle A^{rs}) = 0$.
Here we state Proposition \ref{Prop:UnstableHolonomyBounds} and Lemma \ref{juldens}, that will be proved in Sections \ref{SEC:Internested property of the $rsc$-juliennes} and \ref{Section: Density of rsc-julienne}, respectively.

\begin{prop}\label{Prop:UnstableHolonomyBounds}
There exists a positive integer $k$ such that for any $x,x'\in \mathcal{W}^u(p,1)$, 
the holonomy map
$h^u\colon \widehat{\mathcal{W}}^{rsc}(x)\to\widehat{\mathcal{W}}^{rsc}(x')$
induced by the unstable foliation $\mathcal{W}^u$ satisfies that for all $n\geq k$,
$$\widehat{J}^{rsc}_{n+k}(x') \subset  h^u(\widehat{J}^{rsc}_{n}(x)) \subset \widehat{J}^{rsc}_{n-k}(x').$$
\end{prop}
Not only are the sequences $h^u(\widehat{J}^{rsc}_{n}(x))$ and $\widehat{J}^{rsc}_{n-k}(x')$ internested, but $k$ does not depend on the choice of $x$ and $x'$. 

\begin{lema} \label{juldens}
    A point $x$ in $M_f$ is an $rsc$-julienne density point of $A^u$ if and only if $\pi(x)$ is a Lebesgue density point of $\pi(A^{rs})$.
\end{lema}

To prove Lemma \ref{lemma:A0u-sat}, we will use the following.

\begin{lema} \label{Matrix24}
    The set of $rsc$-julienne density points of $A^u$ is $u$-saturated.
\end{lema}
\begin{proof}
    Let $x,x'\in \mathcal{W}^u(p,1)$ and $h^u\colon \widehat{\mathcal{W}}^{rsc}(x)\to\widehat{\mathcal{W}}^{rsc}(x')$ be
    the holonomy map
    induced by the unstable foliation $\mathcal{W}^u$. Then transverse absolute continuity of $h^u$ with bounded Jacobians implies that
    $$\lim_{n\to \infty}  \widehat{\nu}^{rsc} (A^u:\widehat{J}_n^{rsc}(x) ) =1 \iff 
    \lim_{n\to \infty} \widehat{\nu}^{rsc} (h^u(A^u): h^u (\widehat{J}_n^{rsc}(x)) ) =1.$$
    Note that $ \widehat{\nu}^{rsc} (h^u(A^u): h^u (\widehat{J}_n^{rsc}(x)) )= \widehat{\nu}^{rsc} (A^u: h^u (\widehat{J}_n^{rsc}(x)) )$, since $A^u$ is $u$-saturated.

    Working inside of $\widehat{\mathcal{W}}^{rsc}(x')$ and with respect to the measure $\widehat{\nu}^{rsc} $, we will apply Lemma \ref{comparablesequences} to the sequences $h^u (\widehat{J}_n^{rsc}(x))$ and $\widehat{J}_n^{rsc}(x') $.
    Proposition \ref{Prop:UnstableHolonomyBounds} implies that the sequences $h^u (\widehat{J}_n^{rsc}(x))$ and $\widehat{J}_n^{rsc}(x') $ are internested. 
    Lemma \ref{comparablesequences} now tells us that 
    $$
    \lim_{n\to \infty} \widehat{\nu}^{rsc} (A^u: h^u (\widehat{J}_n^{rsc}(x)) ) =1
    \iff 
    \lim_{n\to \infty}  \widehat{\nu}^{rsc} (A^u:\widehat{J}_n^{rsc}(x') ) =1 ,
    $$
    the result follows from this.
\end{proof}

Now we can show the $u$-saturation of the set of Lebesgue density points of $A_0$,
which easily follows from Lemma \ref{juldens} and Lemma \ref{Matrix24}.
\begin{proof}[Proof of Lemma \ref{lemma:A0u-sat}]
Let $x_0,x_0'\in M$ be such that $x_0, x'_0 \in \pi(\mathcal{W}^u(p,1))$ and $x_0$ is a Lebesgue density point of $A_0$.
Then there exists $x,x' \in \mathcal{W}^u(p,1)\subset M_f$ satisfying $\pi(x)=x_0$ and $\pi(x')=x_0'$.

Since $\nu(A^u\triangle A^{rs})=0$, Lemma \ref{juldens} implies that the $rsc$-julienne density points of $A^u$ are precisely the points that project down to the Lebesgue density points of $A_0$.
Since $x$ projects down to a Lebesgue density point of $A_0$ we see that $x$ is an $rsc$-julienne density point of $A^u$.
From Lemma \ref{Matrix24} it follows that $x'$ is an $rsc$-julienne density point of $A^u$. Then applying Lemma \ref{juldens} again we see that $x_0'$ is a Lebesgue density point of $A_0$.
Thus, the set of Lebesgue density points of $A_0$ is $u$-saturated.
\end{proof}

\section{Internested property of the \textit{rsc}-juliennes}
\label{SEC:Internested property of the $rsc$-juliennes}
We adapt the proof from \cite[Section 6]{BW10} to establish Proposition \ref{Prop:UnstableHolonomyBounds}.
Recall that
$h^u\colon \widehat{\mathcal{W}}^{rsc}(x)\to\widehat{\mathcal{W}}^{rsc}(x')$ 
is
the holonomy map
induced by the unstable foliation $\mathcal{W}^u$.
It suffices to show that we can choose a positive integer $k$ such that
\begin{equation} \label{eq:HolonomyBound}
    h^u(\widehat J_n^{rsc}(x)) \subset \widehat J_{n-k}^{rsc}(x')
\end{equation}
for all $n\geq k$, with $k$ being independent of $x,x'\in\mathcal{W}^u(p,1)$.
In \cite{BW10}, Burns and Wilkinson worked with the stable holonomy and the fake stable holonomy. Here, we adapt their approach to the unstable holonomy and the fake unstable holonomy. 

\begin{lema}\label{Lemma31}
    There exists a positive integer $k_1$ such that the following holds for every integer $n\geq k_1$,
    \[
    \widehat{h}^u(\widehat{B}^c_n(x)) \subset 
    \widehat{B}^c_{n-k_1}(x').
    \]
    for all $x,x'\in \mathcal{W}^s(p)$,  where $\widehat{h}^u\colon \widehat{\mathcal{W}}^c\loc(x)\to \widehat{\mathcal{W}}^c(x')$ is the local $\widehat{\mathcal{W}}^u$ holonomy.
\end{lema}
\begin{proof}
    This is an adaptation of Lemma 6.1 in \cite{BW10}.
    Our Proposition \ref{prop:fakefoln} implies $\widehat{h}^u$ is Lipschitz
    and
    the result follows from this.
\end{proof}

\begin{figure}[ht]
    \centering
\begin{tikzpicture}

\draw[name path=trueunstable_p,red, thick] (-1,-1) .. controls +(1,-0.6) and  +(-1,0.5) .. node[midway] (p) {} (10,-1.6) node[near start] (x) {} node[near end] (xprime) {} node[pos=.1,sloped, below] {$\mathcal{W}^u(p,1)$};

\coordinate (q) at ($(x)+(0,1)$) ;
\coordinate (qprime) at ($(xprime)+(0,1)$) ;
\draw[color=green!50!black,thick] (x.center) 
--
(q) node[midway, left] {$\widehat{\mathcal{W}}^c(x)$};
\draw[name path=fake, magenta,thick] (q) .. controls +(1,-0.2) and  +(-1,0.2).. (qprime) node[midway, above,sloped] {$\widehat{\mathcal{W}}^u$};
\draw[|->, magenta,thick] ($(q)+(0,1)$) .. controls +(1,-0.2) and  +(-1,0.2).. ($(qprime)+(0,1)$) node[midway, above,sloped] {$\widehat{h}^u$};
\draw[name path=fakecenter_xprime,color=green!50!black,thick] (xprime.center) 
--
(qprime) node[midway, right] {$\widehat{\mathcal{W}}^c(x')$};




\node [anchor=north] at (x) {$x$};
\node [anchor=north] at (xprime) {$x'$};

\node[circle,inner sep=1.5pt,fill=black] at (p) {}; 
\node[below] at (p) {$p$};
\fill (xprime) circle [radius=2pt];
\fill (x) circle [radius=2pt];
\end{tikzpicture}
    \caption{
    The set $\mathcal{W}^u(p,1)$ is a subset of $\widehat{\mathcal{W}}^u(p)$.
    Thus, in this figure, $\mathcal{W}^u(p,1)$ depicts both the unstable leaf and the fake unstable leaf through $p$.
    Moreover,  $\widehat{h}^u(x)=x'$.
    }
\end{figure}
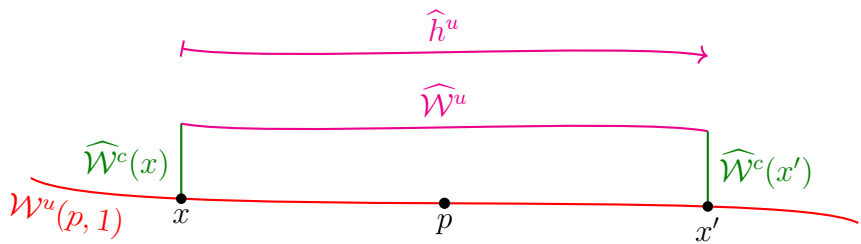
In \cite[Subsection 4.2]{BW10}, Burns and Wilkinson introduce thin neighbourhoods of $\mathcal{W}^s(p,1)$, which are a sequence of sets $S_n$ containing $\mathcal{W}^s(p,1)$.
In our proof, we swap the roles of $s$ and $u$.
Hence, we define the \textit{thin neighbourhoods of $\mathcal{W}^u(p,1)$} by
\[ 
U_n \coloneqq \bigcup_{q\in \mathcal{W}^u(p,1)} \widehat{B}^c_n(q) . 
\]
We also define a sequence of juliennes in the stable directions $r$ and $s$ by
\[
\widehat{J}^{rs}_n(x)\coloneqq \bigcup_{q\in \widehat{J}^{s}_{n}(x)}J^r_n(q).
\]

\begin{lema}\label{Lemma32}
There exists a positive integer $k_2$ such that the following holds for every integer $n\geq k_2 $. 
Suppose $q\in U_n$, $q' \in U_{n-{k_1}}$, and $q'\in \widehat{\mathcal{W}}^u(q)$.
Let $y\in \widehat{J}^{rs}_n(q)$, and let $y'$ be the image of $y$ under the $\mathcal{W}^u$ holonomy from $\widehat{\mathcal{W}}\loc^{rsc}(q)$ to $\widehat{\mathcal{W}}^{rsc}(q')$.

Then $y'\in\widehat{J}^{rs}_{n-k_2}(z')$ for some $z'\in \widehat{B}^c_{n-k_2}(q')$. 
\end{lema}

\begin{figure}[ht]
    \centering
\begin{tikzpicture}
\draw[name path=trueunstable_p,red, thick] (-1,-1) .. controls +(1,-0.6) and  +(-1,0.5) .. node[midway] (p) {} (10,-1.6) node[near start] (x) {} node[near end] (xprime) {} node[pos=.1,sloped, below] {$\mathcal{W}^u(p,1)$};

\coordinate (q) at ($(x)+(0,1)$) ;
\coordinate (qprime) at ($(xprime)+(0,1)$) ;
\coordinate (qsec) at ($(xprime)+(0,2)$) ;
\draw[color=green!50!black,thick] (x.center) -- (q) node[midway, right] {$\widehat{\mathcal{W}}^c(x)$};
\draw[name path=fake, magenta,thick] (q) .. controls +(1,-0.2) and  +(-1,0.2).. (qprime) node[midway, above,sloped] {$\widehat{\mathcal{W}}^u$};
\draw[name path=fakecenter_xprime,color=green!50!black,thick] (xprime.center) -- (qprime) node[midway, right] {$\widehat{\mathcal{W}}^c(x')$};

\draw[cyan!50!black,thick] (q) .. controls +(0,0.2) and  +(.1,0) ..  (0,1) node[midway,left] {$\widehat{\mathcal{W}}^s$};

\draw[red, thick] (0,2) .. controls +(1,-0.3) and  +(-1,0.2) .. (7,2.5) node[midway, above,sloped] {$\mathcal{W}^u$};
\draw[|->,red, thick] (0,3) .. controls +(1,-0.3) and  +(-1,0.2) .. (7,3.5) node[midway, above,sloped] {$h^u$};

\draw[densely dashed, color=blue!30!black,thick] (0,2) -- (0,1) node[midway,left] {$\mathcal{W}^r$};

\fill (qprime) circle [radius=2pt];
\fill (q) circle [radius=2pt];
\node [anchor=north] at (x) {$x$};
\node [anchor=east] at (q) {$q$};
\node [anchor=north] at (xprime) {$x'$};
\node [anchor=west] at (qprime) {$q'$};

\node[circle,inner sep=1.5pt,fill=black] at (p) {}; 
\node[below] at (p) {$p$};
\fill (xprime) circle [radius=2pt];
\fill (x) circle [radius=2pt];
\coordinate (y) at (0,2);
\coordinate (y1) at (7,2.5);
\node[circle,inner sep=1.5pt,fill=black] at (y) {}; 
\node[circle,inner sep=1.5pt,fill=black] at (y1) {}; 
\node [above] at (y) {$y$};
\node [above] at (y1) {$y'$};
\draw[->,thick] (4.5,-2) -- (4.5,-3) node[midway,right] {$\sigma^{-n}$};
\begin{scope}[shift={(-1,2)}]

\draw[red, thick] (3,-6.5) .. controls +(.2,-0.1) and  +(-.2,0.1) .. (5,-6) node[midway, above,sloped] {$O(\lambda^u_n)$};
\draw[densely dashed, color=blue!30!black,thick] (3,-6.5) -- (3,-8.5) node[midway,right] {$<\tau_n$};
\draw[densely dashed, color=blue!30!black,thick] (5,-6) -- (5,-8) node[midway,right] {$<\tau_n$};
\draw[cyan!50!black,thick] (6,-10.5) .. controls +(0,0.2) and  +(.1,0) ..  (3,-8.5) node[midway,right] {$<\tau_n$};
\coordinate (sigmaz) at (7.5,-9);
\draw[cyan!50!black,thick] (sigmaz) .. controls +(0,0) and  +(.1,0) ..  (5,-8) node[pos=.4,above,sloped] {$O(\tau_n)$};
\coordinate (sigmaq1) at ($(sigmaz)+(0,-1)$);

\coordinate (sigmay1) at (5,-6);
\fill (sigmay1) circle [radius=2pt];
\node [right] at (sigmay1) {$\sigma^{-n}y'$};
\draw[magenta, thick] (6,-10.5) .. controls +(.1,-0.1) and  +(-.1,0.1) .. (sigmaq1) node[midway, below,sloped] {$O(\lambda^u_n)$};
\draw[green!50!black,thick] (sigmaq1) -- (sigmaz) node[midway,right] {$O(\tau_n)$};    
\coordinate (w1) at (5,-8);
\node [above right] at (w1) {$\sigma^{-n}w'$};
\fill (w1) circle [radius=2pt];
\fill (6,-10.5) circle [radius=2pt];
\node [left] at (6,-10.5) {$\sigma^{-n}q$};
\fill (3,-6.5) circle [radius=2pt];
\node [left] at (3,-6.5) {$\sigma^{-n}y$};
\fill (sigmaq1) circle [radius=2pt];
\node [below right] at (sigmaq1) {$\sigma^{-n}q'$};
\fill (sigmaz) circle [radius=2pt];
\node [right] at (sigmaz) {$\sigma^{-n}z'$};
\coordinate (w) at (3,-8.5);
\node [left] at (w) {$\sigma^{-n}w$};
\fill (w) circle [radius=2pt];
\end{scope}

\draw[->,thick] (4.5,-9) -- (4.5,-10) node[midway,right] {$\sigma^{n}$};

\begin{scope}[shift={(0,-12.5)}]
\draw[draw=none] (-1,-1) .. controls +(1,-0.6) and  +(-1,0.5) .. node[midway] (p) {} (10,-1.6) node[near start] (x) {} node[near end] (xprime) {};

\coordinate (q) at ($(x)+(0,1)$) ;
\coordinate (qprime) at ($(xprime)+(0,1)$) ;
\coordinate (qsec) at ($(xprime)+(0,2)$) ;
\draw[name path=fake, magenta,thick] (q) .. controls +(1,-0.2) and  +(-1,0.2).. (qprime);
\draw[name path=fakecenter_xprime,color=green!50!black,thick] (qprime) -- (qsec) node[midway,right] {$O((\widehat{\lambda}^c_n)^{-1}\tau_n)$};

\draw[cyan!50!black,thick] (q) .. controls +(0,0.2) and  +(.1,0) ..  (0,1);
\draw[cyan!50!black,thick] (qsec) .. controls +(0,0.1) and  +(.1,0) ..  (7,1.5) ;

\draw[red, thick] (0,2) .. controls +(1,-0.3) and  +(-1,0.2) .. (7,2.5);

\draw[densely dashed, color=blue!30!black,thick] (0,2) -- (0,1);
\draw[densely dashed, color=blue!30!black,thick] (7,2.5) -- (7,1.5) ;

\fill (qprime) circle [radius=2pt];
\fill (q) circle [radius=2pt];

\coordinate (y) at (0,2);
\coordinate (y1) at (7,2.5);
\node[circle,inner sep=1.5pt,fill=black] at (y) {}; 
\node[circle,inner sep=1.5pt,fill=black] at (y1) {}; 
\node[circle,inner sep=1.5pt,fill=black] at (qsec) {}; 

\node [below] at (q) {$q$};
\node [below] at (qprime) {$q'$};
\node [above right] at (qsec) {$z'$};
\node [above] at (y1) {$y'$};
\node [above] at (y) {$y$};
\end{scope}
\end{tikzpicture}
    \caption{Picture for the proof of Lemma \ref{Lemma32}.}
    \label{Fig: Lemma32}
\end{figure}
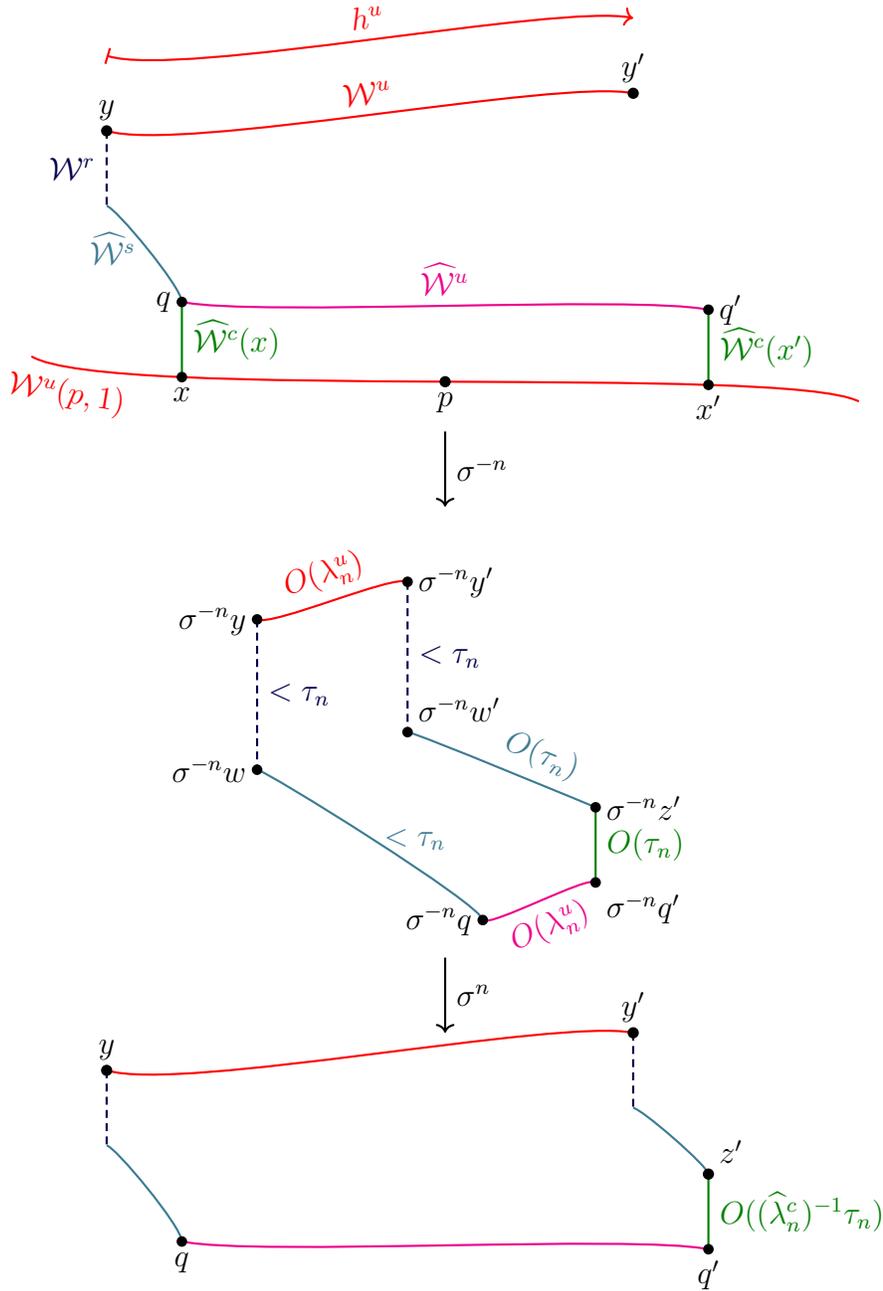

Our proof of this lemma closely follows the proof of \cite[Lemma 6.2]{BW10}.
We provide a detailed version in Appendix \ref{Appendix proofs},
but we sketch the proof here.
Figure \ref{Fig: Lemma32} shows the setup for Lemma \ref{Lemma32} and illustrates the key step of the proof;
studying the distances between the images of the points under $\sigma^{-n}$.
We want to bound the growth of $d(q',z')$, $d(\sigma^{-n}z',\sigma^{-n}w')$, and $d(\sigma^{-n}w',\sigma^{-n}y')$.
Our approach to dealing with $d(q',z')$ and $d(\sigma^{-n}z',\sigma^{-n}w')$ is similar to the proof of \cite[Lemma 6.2]{BW10}.
Since the points $\sigma^{-n}w'$ and $\sigma^{-n}y'$ are in the same local $r$-set, with the $r$-direction being introduced by the noninvertible case, bounding the distance $d(\sigma^{-n}w',\sigma^{-n}y')$ presents a new problem. 
The points $\sigma^{-n}w'$ and $\sigma^{-n}y'$ are the images of the points $\sigma^{-n}w$ and $\sigma^{-n}y$ through the $scu$-holonomy that sends $\mathcal{W}^r\loc(w)$ to $\mathcal{W}^r\loc(w')$. From Lemma \ref{lemma:fibermetricinvariance}, it follows that $d(\sigma^{-n}w',\sigma^{-n}y')<\tau_n$.

Proposition \ref{Prop:UnstableHolonomyBounds} follows by combining
Lemmas \ref{Lemma31} and \ref{Lemma32}.
See Appendix \ref{Appendix proofs} for a detailed proof.

\FloatBarrier

\section{Density of \textit{rsc}-juliennes}\label{Section: Density of rsc-julienne}

We now come to the proof of Lemma \ref{juldens}. 
Recall that $A_0\subset M$ is an $s$-saturated set and $A^{rs}=\pi^{-1}(A_0)$.
Moreover, $A^{rs}\subset M_f$ is both an $rs$-saturated set and an essentially $u$-saturated set, with $A^u$ being the $u$-saturate of $A^{rs}$.
As defined previously, for $\alpha\in \{s,c,u,sc,cu\}$, $\widehat{\mathcal{W}}^\alpha$ is the fake invariant foliation
defined near the orbit of $p$.

The results in this section establish the following chain of equivalences.
\begin{eqnarray*}
\pi(x)\text{ is a Lebesgue density point of } \pi(A^{rs}) & \iff & \lim_{n\to \infty} \nu (A^{rs}:B_n(x))  =1 \\
\iff \lim_{n\to \infty} \nu (A^{rs}:E_n(x))  =1 & \iff      & \lim_{n\to \infty} \nu (A^{rs}:F_n(x))  =1\\
\iff \lim_{n\to \infty} \nu (A^{rs}:G_n(x))  =1 & \iff      & \lim_{n\to \infty} \widehat{\nu}^{rsc} (A^u:\widehat{J}^{rsc}_n(x))  =1.
\end{eqnarray*}

In \cite[Section 8]{BW10}, Burns and Wilkinson define seven sequences of sets $B_n(x)$, $C_n(x)$, $D_n(x)$, $E_n(x)$, $F_n(x)$, $G_n(x)$, and $\widehat{J}^{cu}_n(x)$.
These sequences of sets serve to prove that for a measurable set $X$ that is both $s$-saturated and essentially $u$-saturated, a point $x\in \mathcal{W}^s(p,1)$ is a Lebesgue density point of $X$ if and only if $$\lim_{n\to \infty} \widehat{m}_{cu} (X:\widehat{J}^{cu}_n(x))  =1.$$
They do this by establishing the following chain of equivalences.
\begin{eqnarray*}
x\text{ is a Lebesgue density point of } X & \iff & \lim_{n\to \infty} m (X:B_n(x))  =1 \\
\iff \lim_{n\to \infty} m (X:C_n(x))  =1 & \iff      & \lim_{n\to \infty} m (X:D_n(x))  =1\\
\iff \lim_{n\to \infty} m (X:E_n(x))  =1 & \iff      & \lim_{n\to \infty} m (X:F_n(x))  =1\\
\iff \lim_{n\to \infty} m (X:G_n(x))  =1& \iff      & \lim_{n\to \infty} \widehat{m}_{cu} (X:\widehat{J}^{cu}_n(x))  =1.
\end{eqnarray*}

We define our versions of $B_n(x)$, $E_n(x)$, $F_n(x)$, $G_n(x)$, and $\widehat{J}^{rsc}_n(x)$,
closely following the definitions of the respective sequences of sets in \cite{BW10},
but extend the juliennes a small distance into the fiber or ``$r$'' direction.
Our equivalence
$$\lim_{n\to \infty} \nu (A^{rs}:B_n(x))  =1 \iff \lim_{n\to \infty} \nu (A^{rs}:E_n(x))  =1,$$
follows almost directly from their corresponding results.
Hence, we do not define versions of $C_n(x)$ and $D_n(x)$ in our proof.
The primary novelty in the set definitions and the proofs of the equivalences included in this section is our approach to handling the $r$-direction.

Define 
$$B_n^{scu}(x) \coloneqq \mathcal{W}^{scu}(x,\delta_n)
\quad\text{and} \quad
B_n(x) \coloneqq  \bigcup_{q\in B_n^{scu}(x)} \mathcal{W}^r\loc(q),$$
then $B_n(x)=\pi^{-1}(B(x_0,\delta_n))$. 
Since $\mu$ is the push-forward measure $\pi_*(\nu)$,
we have $\mu({A_0}:B(x_0,\delta_n))={\nu({A^{rs}}:B_n(x))}$.
Then it follows that $\pi(x)$ is a Lebesgue density point of $\pi(A^{rs})$ if and only if $$\lim_{n\to \infty} \nu (A^{rs}:B_n(x))  =1.$$

Before showing the equivalence with $\lim_{n\to \infty} m (X:E_n(x)) = 1$, we establish two important properties of the $J^r_n$ juliennes.
Recall that $J^r_n(y) $ is defined as $ \sigma^n(\mathcal{W}^r(\sigma^{-n}y,\tau_n))$.
\begin{lema}\label{lemma:Jr-regular}
    The $\nu^r$-measure of $J^r_n(y)$ has the following properties:
    \begin{enumerate}
        \item $\nu^r(J^r_n(y))$ does not depend on $y\in M_f$, and,
        \item $J^r_n(y)$ is regular.
    \end{enumerate}
\end{lema}
\begin{proof}
    Lemma \ref{lemma:fibermetric} and Proposition \ref{Prop:r-measure} establish that for all $y,q\in M_f$ such that $d(y,q)\leq \lambda^r_m$,
    \[
    q\in \sigma^{m-1}(\mathcal{W}^r\loc(\sigma^{-(m-1)}y))\quad\text{and}\quad\nu^r(\sigma^{m-1}(\mathcal{W}^r\loc(\sigma^{-(m-1)}y)))=\frac{1}{\deg(f)^{m-1}}.
    \]
    Thus the $\nu^r$ measure of a set $\mathcal{W}^r(y,r)$ only depends on $r$.
    Note that $J^r_n(y)=\mathcal{W}^r(y,\tau_n\lambda_n^r)$, the first result follows from this.

    To show that $J^r_n(y)$ is regular note that there exists a positive integer $K$ such that $\tau\geq\lambda^r_K$.
    Then
    \[
    \nu^r(\mathcal{W}^r(y,\tau_{n+1}))\geq \nu^r(\mathcal{W}^r(y,\tau_{n}\lambda^r_K))=
    \deg(f)^{-K}\cdot \nu^r(\mathcal{W}^r(y,\tau_{n})). \qedhere
    \]
\end{proof}
Recall from Section \ref{Sec: Unstable saturation} that $\widehat{B}^c_n(q)$ is the Riemannian ball inside the fake center manifold through $q$ centred at $q$ of radius $\delta_n$.
Define \[
E_n^{scu}(x)\coloneqq \bigcup_{y\in\widehat{\mathcal{W}}^u(x,\delta_n)} \;\bigcup_{z\in\widehat{B}^c_n(y)} \widehat{J}^s_n(z)
\quad\text{and} \quad
E_n(x) \coloneqq  \bigcup_{q\in E^{scu}_n(x)} J^r_n(q).
\]
\begin{lema}\label{Lemma:BscutoEscu}
The following equivalence holds,
$$\lim_{n\to \infty} \nu^{scu} (A^{rs}:B_n^{scu}(x))  =1 \iff 
\lim_{n\to \infty} \nu^{scu} (A^{rs}:E_n^{scu}(x))  =1.$$
\end{lema}
\begin{proof}
The definitions of $B_n^{scu}(x)$ and $E_n^{scu}(x)$ are analogous to the sequences of sets $B_n(x)$ and $E_n(x)$ from \cite[Section 8]{BW10}, respectively. 
We swap the roles of $\widehat{\mathcal{W}}^u$ and $\widehat{\mathcal{W}}^s$ since we are showing that the set of Lebesgue density points of $A_0$ is $u$-saturated. The proof of Lemma \ref{Lemma:BscutoEscu} is identical to the steps to prove the equivalence between $B_n(x)$ and $E_n(x)$ carried out by Burns and Wilkinson, accounting for the swapping of unstable and stable directions.    
\end{proof}

\begin{lema}
    The following equivalence holds,
    \[
    \lim_{n\to \infty} \nu (A^{rs}:B_n (x))  =1 \iff 
    \lim_{n\to \infty} \nu (A^{rs}:E_n(x))  =1.
    \]
\end{lema}
\begin{proof}
    Proposition \ref{defnu} states that, locally, the measure $\nu$ is the product measure $\nu^{scu}\times \nu^{r}$.
    The sets $B_n(x)$ and $E_n(x)$ can be expressed as a product with respect to the identification obtained from Proposition \ref{canon}.
    These sets are given as $B_n^{scu}(x)\times \mathcal{W}^r\loc(x)$ and $E_n^{scu}(x)\times J_n^r(x)$, respectively. Since $A^{rs}$ is $r$-saturated the result follows from Lemma \ref{Lemma:BscutoEscu}.
\end{proof}

Let $F_n(x)$ be defined as the product of foliations between $\widehat{J}_n^{rsc}(x)$ and $\mathcal{W}^u(x,\delta_n)$ as follows
$$
F_n(x)\coloneqq
\bigcup_{\substack{{q\in \widehat{J}^{rsc}_n(x)} \\{ q' \in \mathcal{W}^u(x,\delta_n)}}} \mathcal{W}^u\loc(q) \cap \widehat{\mathcal{W}}^{rsc}\loc (q').
$$
Note that for $n$ large enough Proposition \ref{prop:fakefoln}, part \ref{Lemma_i:Almost tangency to invariant distributions} implies that $\mathcal{W}^u\loc(q) $ and $ \widehat{\mathcal{W}}^{rsc}\loc (q')$ are transverse.
Define 
\[
G_n(x)\coloneqq \bigcup_{q \in \widehat{J}^{rsc}_n(x)} \mathcal{W}^u(q,\delta_n) .
\]
We can also write this as
\[
    G_n(x)=\bigcup_{z\in J^r_{n}(x) } G^{scu}_n(z)
    \quad \text{where} \quad
    G^{scu}_n(z)\coloneqq  \bigcup_{q \in \widehat{J}^{sc}_n(z)} \mathcal{W}^u(q,\delta_n).
\]

\begin{lema}
\label{Lemma E F G internested}
The sequences $E_n(x)$ and $F_n(x)$ are internested, and the sequences $F_n(x)$ and $G_n(x)$ are internested.
\end{lema}
The proof of Lemma \ref{Lemma E F G internested} is a direct adaptation of \cite[Lemma 8.2.]{BW10}
with the details given in Appendix \ref{Appendix proofs}.

\begin{lema}
\label{Lemma Gn is regular}
For each $x \in \mathcal{W}^u(p,1)$, the sequence $G_n(x)$ is regular.
\end{lema}
\begin{proof}
From \cite[Lemma 8.3]{BW10} we obtain that for each $z\in J^r_{n}(x)$,
the sequence $G^{scu}_n(z)$ is regular with respect to the $\nu^{scu} $ measure.
Proposition \ref{prop:fakefoln},
part \ref{Lemma_viii:Endos} implies that $\pi( \widehat{J}^{cs}_n(z))$ does not depend on the choice of $z\in J^r_{n}(x)$.
Then from the continuity $E^u$ we get that there exists a $C>0$ and a $K\geq 1$, such that for all $n\geq 0$ and any $z\in J^r_{n}(x)$,
we have that $\nu^{scu}(G^{scu}_{n+K}(z))\geq C \nu^{scu}(G^{scu}_{n}(z)).$
As shown in Lemma \ref{lemma:Jr-regular},
the sequence $J^r_{n}(x)$ is regular with respect to the measure $\nu^r$.
The result follows from this, as $\nu$ is locally $\nu^r\times\nu^{scu}$.
\end{proof}

From Lemmas \ref{Lemma E F G internested} and \ref{Lemma Gn is regular} we get the following equivalences
$$\lim_{n\to \infty} \nu (A^{rs}:E_n(x))  =1 
\iff \lim_{n\to \infty} \nu (A^{rs}:F_n(x))  =1 
\iff \lim_{n\to \infty} \nu (A^{rs}:G_n(x))  =1.$$
\begin{lema}
    The following equivalence holds,
    $$\lim_{n\to \infty} \nu (A^{rs}:G_n(x))  =1  \iff \lim_{n\to \infty} \widehat{\nu}^{rsc} (A^u: \widehat{J}^{rsc}_n(x))  =1.$$
\end{lema}
\begin{proof}
Since $\nu (A^{rs}\triangle A^{u })=0$ we have,
    $$\lim_{n\to \infty} \nu (A^{rs}:G_n(x))  =1  \iff\lim_{n\to \infty} \nu (A^{u}:G_n(x))  =1 .$$
Since $E^u(x)$ depends continuously on $x\in M_f$ there exists $c\geq 1$ such that for all $n\geq 0$ and any $q,q'\in \widehat{J}^{rsc}_n(x)$, the following inequalities hold
    $$ c^{-1}\leq \frac{\nu^u(\mathcal{W}^u(q,\delta^n))}{\nu^u(\mathcal{W}^u(q',\delta^n))}\leq c.$$\par 
The proof of \cite[Proposition 2.7]{BW10} is entirely measure theoretic in nature and applies to our setting.
Thus,
since $A^u$ is $\mathcal{W}^u$-saturated the result follows from \cite[Proposition 2.7]{BW10}. 
\end{proof}

\begin{lema}
For each $x\in \mathcal{W}^u(p,1)$,
the sequence $\widehat{J}_n^{rsc}(x)$ is regular
\end{lema}
\begin{proof}
Adapting Section 7 of \cite{BW10} we can prove that there exists a $c>0$ such that $\widehat{\nu}^{sc}(\widehat{J}_{n+1}^{sc}(x))>c\cdot\widehat{\nu}^{sc}(\widehat{J}_{n}^{sc}(x))$.
Since the measure \( \widehat{\nu}^{rsc} \) is the product \( \nu^r \times \widehat{\nu}^{sc} \), and $J^r_n(x))$ is regular, the desired result follows directly.
\end{proof}

In this section, we established that
a point $\pi (x)$ is a Lebesgue density point of $\pi(A^{rs})$ if and only if $$\lim_{n\to \infty} \widehat{\nu}^{cu} (A^{u}:\widehat{J}^{rsc}_n(x))  =1,$$
where $A^u$ is the $u$-saturate of $A^{rs}$.
To do this, we first establish the equivalence between $\pi (x)$ being a Lebesgue density point of $\pi(A^{rs})$ and 
$$\lim _{n\to \infty} \nu (A^{rs}:B_n(x))  =1.$$
This step changes the setting from $M$ to $M_f$, $\mu$ to $\nu$ and, $\pi(A^{rs})$ to $A^{rs}$. Then using the results from \cite[Section 8]{BW10} we proved the equivalence with $$\lim _{n\to \infty} \nu (A^{rs}:E_n(x))  =1.$$
Using Proposition \ref{Prop:UnstableHolonomyBounds} we get that  \( E_n(x) \), \( F_n(x) \), and \( G_n(x) \) are internested which allows us to get the corresponding equivalence with the sequence of sets $F_n(x)$ and later $$\lim _{n\to \infty} \nu (A^{rs}:G_n(x))  =1.$$
Since $A^{rs}$ is both $rs$-saturated and essentially $u$-saturated, and $A^u$ is the $u$-saturate of $A^{rs}$ this is equivalent to
$$\lim _{n\to \infty} \nu (A^{u}:G_n(x))  =1.$$
The sets of sequence $G_n(x)$ were defined as the union over \( \widehat{J}^{rsc}_n(x) \), in order to prove that $G_n(x)$ is regular we proved that the fibres of $G_n(x)$ over \( \widehat{J}^{rsc}_n(x) \) are uniform. The final equivalence follows from this, and by establishing the regularity of \( \widehat{J}^{rsc}_n(x) \), we completed the proof.
\appendix
\section{K-property}\label{appendix:K-property}
We now establish that $\sig : M_f \to M_f$ has the K-property
with respect to $\nu.$
This is proved by an adaptation of standard arguments for
expanding and contracting foliations.
See
\cite{climenhaga2012measuretheorydynamicaleyes}
or \cite{ponce2019introduction}
for more details.

Fix a constant $0 < \lam < 1$ such that if
$x, y \in M_f$ are distinct points on the same unstable leaf, then
$d_u( \sig \inv(x), \sig \inv(y)) < \lam d_u(x,y).$
Let $B \subset M$ be a small geometric open ball in the base manifold $M$
and define a measurable function $r : M_f \to [0, \infty),$ by
\[
    r(x) = \inf_{n \ge 0} \lam^{-n} \dist( x_{-n}, \del B)
\]
where $x = \{x_n\}_{n \in \bbZ} \in M_f.$

\begin{lema} \label{lemma:almostpositive}
    The function $r$ is positive $\nu$-almost everywhere.
\end{lema}
\begin{proof}
    For $n \ge 0,$ define $V_n \subof M$ by
    $p \in V_n$ if and only if $\dist(p, \del B) < \lam^n.$
    Since the ball $B$ has a smooth boundary, the volume
    of $V_n$ is proportional to $\lam^n$
    and so the series
    $\sum_{n=0}^\infty \mu(V_n)$
    converges.
    If we write $X_n = \sig^n \pi \inv(V_n),$
    then $\sum_{n=0}^\infty \nu(X_n)$ converges to same limit
    and $x \in X_n$ holds if and only if 
    $\lam^{-n} \dist( x_{-n}, \del B) < 1.$
    The result can then be proved using the Borel--Cantelli lemma.
\end{proof}
Let $U = \pi \inv(B)$ and define a partition $\xi$ of $M_f$ as follows:
If $x \in M_f \sans U,$ then $\xi(x) = M_f \sans U.$
If $x \in U,$ then $\xi(x)$ is the connected component
of $W^u(x) \cap U$ containing $x.$
In the second case, the diameters of the resulting
unstable plaques are uniformly bounded; that is, there is $D > 0$
such that if $x, y \in U$ and $\xi(x) = \xi(y),$ then $d_u(x,y) < D.$
From $\xi,$ define a finer partition
$\displaystyle{ \eta = \bigvee_{n \ge 0} \sig^n \,\xi.}$

\begin{lemma} \label{lemma:leafsubordinate}
    For $\nu$-almost every $x \in M_f,$
    $\eta(x)$ is a subset of the unstable leaf $\Wu(x)$
    and contains an open neighborhood of $x$ inside $\Wu(x).$
\end{lemma}
\begin{proof}
    This follows from Lemma \ref{lemma:almostpositive}.
    For each point $x \in M_f,\,$ $\eta(x)$ contains the
    local unstable manifold of radius $r(x)$ at $x.$
\end{proof}
\begin{lemma} \label{lemma:trivialpartition}
    The partition
    $\displaystyle{\bigvee_{n \ge 0} \sig \invn\,\eta}$ is trivial.
\end{lemma}    

\begin{proof}
    By the ergodicity of $\sig,$
    $\nu$-almost every point $x \in M_f$ has an orbit that visits
    $U$ infinitely often. If $y \in \Wu(x)$ is distinct from $x,$
    then there is $n \ge 0$ such that
    $d_u( \sig^n(x), \sig^n(y) ) > D$ and
    $\sig^n(x) \in U.$
    This implies that
    $\xi(\sig^n(x)) \ne \xi(\sig^n(y))$
    and therefore
    $\eta(\sig^n(x)) \ne \eta(\sig^n(y)).$
\end{proof}
\begin{lemma} \label{lemma:coarserthanfoln}
    Modulo a null set,
    the partition
    $\displaystyle{\bigwedge_{n \ge 0} \sig^n\, \eta}$
    is coarser
    than the unstable
    foliation $\Wu.$
\end{lemma}
\begin{proof}
    Fix a constant $0 < \bt < 1.$
    Since $r : M_f \to [0, \infty)$
    is positive almost everywhere,
    there is $\delta > 0$ such that
    $\nu \big( \{ x \in M_f : r(x) > \delta \} \big) > \bt.$
    By the invariance of the measure under $\sig,$
    we also have for any $n \ge 0$ that
    \[
        \nu \big( \{ x \in M_f : r(\sig \invn(x)) > \delta \} \big) \  > \ \bt.
    \]
    The result follows from this and the uniform expansion of unstable leaves.
\end{proof}
\begin{cor} \label{cor:pinskerunstable}
    The Pinsker partition is coarser than
    the measurable hull of the unstable foliation:
    $\pi(\sig) < \mathcal{H}(\Wu).$
\end{cor}
\begin{proof}
    This follows from the last three lemmas and classic results on entropy.
    See, for instance, \cite[Section 2.6.1]{ponce2019introduction} or \cite[Theorem 12.1]{Rokhlin_1967}.
\end{proof}
A similar proof shows that $\pi(\sig) < \mathcal{H}(\Ws).$

Accessibility classes in $M_f$ are defined
using the fiber or ``$r$'' direction as well as the $s$ and $u$
directions.
Let $\rho$ be the partition given by the projection $\pi : M_f \to M.$
That is, $\rho(x) = \rho(y)$ if and only if $\pi(x) = \pi(y).$
It is easy to see that $\sig \rho$ is a refinement of $\rho$
and that
$\bigvee_{n \ge 0} \sig^n \rho$ is trivial.
Let $W^r$ denote the partition
$\bigwedge_{n \ge 0} \sig \invn \rho$.
It then follows that the Pinsker partition $\pi(\sig)$ is coarser than
$\Hcal(W^r)$.
All together, we have that
\[
    \pi(\sig) < \Hcal(W^r) \wedge \Hcal(\Ws) \wedge \Hcal(\Wu).
\]
Hence, if $\sig$ is essentially accessible
then the Pinsker partition contains a set of full measure
and this shows that $\sig$ satisfies
the Kolmogorov or K-property.

\section{Technical Details and Proofs}\label{Appendix proofs}
    This appendix provides detailed proofs for those that differ significantly from the originals in \cite{BW10}, as well as detailed versions of the proofs previously given a short explanation.
    
    \begin{proof}[Proof of Proposition \ref{Prop:r-measure}]
    Define the basic sets \(\mathcal{E} = \{\sigma^n(\mathcal{W}^r\loc(x)) : n \in \mathbb{Z}, x \in M_f\}\), and let \(\mathcal{A}^r\) be the algebra they generate. Define the function \(\nu^r_0 : \mathcal{A}^r \to [0, \infty]\) by
    \[
    \nu^r_0(A) = 
    \begin{cases} 
    0, & \text{if } A = \emptyset, \\ 
    \sum_{i=1}^n \deg(f)^{-m_i}, & \text{if } A = \bigsqcup_{i=1}^n \sigma^{m_i}(\mathcal{W}^r\loc(x^i)), \\ 
    \infty, & \text{otherwise}.
    \end{cases}
    \]
    
    We now need to show that the function \(\nu^r_0\) is countably additive.
    
    Define $X$ as the topological space that is the exact same set as $M$, but with the discrete topology.
    Define $X_f\subset X^\Z$ as the space of orbits with the subspace topology relative to $X^\Z$. 
    In other words, $X_f$ is the exact same set as $M_f$, but with a different topology, in particular, every $\mathcal{W}^r\loc(x)=\pi^{-1}(x_0)$ is closed and open in $X_f$ and is homeomorphic to a Cantor set. The shift $\sigma\colon X_f\to X_f$ is still a homeomorphism.

    Basic sets are both open and compact.
    Thus, any countable disjoint union of basic sets must be finite.
    Since \(\nu^r_0\) is finitely additive on basic sets, it extends to countable additivity.
    
    By Carathéodory's Extension Theorem \cite[Theorem 1.14]{folland2013real}, \(\nu^r_0\) extends to \(\nu^r\).
    In order to obtain uniqueness, we restrict the premeasure $\nu^r_0$ to $\mathcal{W}^r(x)$. Since $\mathcal{W}^r(x)= \cup_{n=1}^\infty \sigma^{-n}(\mathcal{W}^r\loc(\sigma^n(x)))  $ and
    \[
    \nu^r_0(\sigma^{-n}(\mathcal{W}^r\loc(\sigma^n(x))))=\deg(f)^{n}<\infty ,
    \]
    the restriction of $\nu^r_0$ to $\mathcal{W}^r\loc(x)$ is $\sigma$-finite. 
    Therefore, 
    restricted to each $\mathcal{W}^r(x)$, the extension of $\nu^r_0$ to $\nu^r$ is unique.
\end{proof}

\begin{proof}[Proof of Proposition \ref{Prop:UnstableHolonomyBounds}]
As noted at the beginning of the section, it suffices to prove the inclusion (\ref{eq:HolonomyBound}). 
For $q\in \widehat{B}^c_n(x)$, let $q'\coloneqq\widehat{h}^u(q)$. Then $q' \in \widehat{B}^c_{n-k_1}(x')$ by Lemma \ref{Lemma31}. Hence $q$ and $q'$ are as in the hypothesis of Lemma \ref{Lemma32} we can apply that lemma to obtain
$$h^u(\widehat{J}^{rcs}_{n}(x))\subset\bigcup_{z\in Q} \widehat{J}^{rs}_{n-k_2}(z),\quad \text{where}\quad Q=\bigcup_{q'\in \widehat{B}^c_{n-k_1}(x')}\widehat{B}^c_{n-k_2}(q').$$
It therefore suffices to find $k\geq k_2$, such that $Q\subset \widehat{B}^c_{n-k}(x')$.
\end{proof}
\begin{proof}[Proof of Lemma \ref{Lemma32}]
    Let $z'$ be such that $\sigma^{-n}z'$ is the unique point in the intersection $\widehat{\mathcal{W}}^c(\sigma^{-n}q') \cap \widehat{\mathcal{W}}^{rs}(\sigma^{-n}y')$. 
    As stated in the proof found in \cite[Section 6]{BW10}, it suffices to show that $d(\sigma^{-n}y',\sigma^{-n}z')=O(\tau_n)$ and $d(z',q')=O(\delta_n)$.
    
    By the definition of \( \widehat{J}^{rs}_n(q) \), there exists $w$ such that $\sigma^{-n}w\in \widehat{\mathcal{W}}^s(\sigma^{-n}q,\tau_n)$ and $\sigma^{-n}y\in {\mathcal{W}}^r(\sigma^{-n}w,\tau_n)$. Thus, we have \( d(\sigma^{-n}y, \sigma^{-n}q) \leq 2 \tau_n \).
    By Proposition \ref{prop:fakefoln}, part \ref{Prop_iii: Exponential growth bounds at local scales} subpart \ref{Prop_iii_b} and since $d(q,q')=O(1)$ we have $d(\sigma^{-n}q,\sigma^{-n}q')=O({\lambda^u_{n}})$. Since $d(y,y')=O(1)$ and $y'$ is the image of $y$ under $\mathcal{W}^u$ holonomy, then $d(\sigma^{-n}y,\sigma^{-n}y')=O(\lambda^u_{n})$.\par
    Recall that from the definition of $\tau$, we have that $\lambda^u<\tau$ and the triangle inequality we have that
    $$ d(\sigma^{-n}q' ,\sigma^{-n}y'  ) \leq d(\sigma^{-n}q' ,\sigma^{-n}q  )+ d(\sigma^{-n}q ,\sigma^{-n}y  )+d(\sigma^{-n}y,\sigma^{-n}y')= O(\tau_n). $$
    From the definition of the distance $d$ we have that $ d_M(\pi\circ \sigma^{-n}q' ,\pi\circ\sigma^{-n}y'  ) = O(\tau_n). $ 
    Uniform transversality of the foliations $\widehat{\mathcal{W}}^s$ and $\widehat{\mathcal{W}}^{c}$ implies that both
    $ d_M( {\pi\circ\sigma^{-n}} y',{\pi\circ\sigma^{-n}} z'  )$ and $ d_M(\pi\circ\sigma^{-n}q',\pi\circ\sigma^{-n}z'  )$ are $ O(\tau_n). $ 
    Since $\sigma^{-n}q'$ and $\sigma^{-n}z'$ are in the same $scu$-sheet $ d_M(\pi\circ\sigma^{-n}q',\pi\circ\sigma^{-n}z'  )=d(\sigma^{-n}q',\sigma^{-n}z')=O(\tau_n)$. 
    
    Let $w'$ be such that $\sigma^{-n}y' \in \mathcal{W}^r (\sigma^{-n}w')$ and $\sigma^{-n}w'\in \widehat{\mathcal{W}}^{s}(\sigma^{-n}z') $. The points $\sigma^{-n}w'$ and $\sigma^{-n}y'$ are the images of the points $\sigma^{-n}w$ and $\sigma^{-n}y$ through the $scu$-holonomy that sends $\mathcal{W}^r\loc(w)$ to $\mathcal{W}^r\loc(w')$. From Lemma \ref{lemma:fibermetricinvariance}, it follows that
    \begin{equation}\label{eq:w',y'}
     d(\sigma^{-n}w',\sigma^{-n}y')<\tau_n.   
    \end{equation}
    
    Since $\pi\circ\sigma^{-n}y'=\pi\circ\sigma^{-n}w'$, then 
    \begin{multline}
    \label{eq:w',z'}
    d( {\sigma^{-n}} w',{\sigma^{-n}} z' )= d_M( {\pi\circ\sigma^{-n}} w',{\pi\circ\sigma^{-n}} z'  )= \\
    d_M( {\pi\circ\sigma^{-n}} y',{\pi\circ\sigma^{-n}} z'  )=O(\tau_n).     
    \end{multline}
    From Equations (\ref{eq:w',y'}) and (\ref{eq:w',z'}) we get $d(\sigma^{-n}y',\sigma^{-n}z')=O(\tau_n)$.\par
    Finally, since $d(\sigma^{-n}q',\sigma^{-n}z')=O(\tau_n)$, and $\sigma^{-n}q'$ and $\sigma^{-n}z'$ lie on the same fake center leaf at distance $O(\tau_n)$, Proposition \ref{prop:fakefoln}, part \ref{Prop_iii: Exponential growth bounds at local scales} subpart \ref{Prop_iii_b} now implies $d(q',z') =O((\widehat{\lambda}^c_n)^{-1} \tau_n)$. By definition of $\tau$ and $\delta$ we have that $\tau<\delta\cdot\hat \lambda^c$. Therefore, $(\widehat{\lambda}^c_n)^{-1} \tau_n<\delta_n$, and $d(q',z')=O(\delta_n)$ as desired.
\end{proof}

\begin{proof}[Proof of Lemma \ref{Lemma E F G internested}]
Both $E_n(x)$ and $F_n(x)$ can be expressed as unions over $\mathcal{W}^u(x,\delta^n)$.
Proposition \ref{prop:fakefoln}, part \ref{Prop_v: Uniqueness}, implies that the set $E_n(x)$ is the union of all the $\widehat{J}^{rsc}_n(q)$ for $q$ in $\mathcal{W}^u(x,\delta^n)$.
The set $F_n(x)$ is the union over all $q$ in $\mathcal{W}^u(x,\delta^n)$ of the image of $\widehat{J}^{rsc}_n(x)$ under the unstable holonomy from $\widehat{\mathcal{W}}^{rcs}\loc(x)$ to $\widehat{\mathcal{W}}^{rsc}(q)$.
It follows from Proposition \ref{Prop:UnstableHolonomyBounds} that the sequences $E_n(x)$ and $F_n(x)$ are internested.\par
To see that $F_n(x)$ and $G_n(x)$ are internested, it will suffice to prove that $k_0$ can be chosen so that for every $q\in \widehat{J}^{rsc}_n(x)$ we have $$\mathcal{W}^u(q,\delta_{n+k_0})\subset h(\mathcal{W}^u(x,\delta_{n}) )\subset\mathcal{W}^u(q,\delta_{n-k_0})$$
where $h$ is the $rsc$-holonomy from $\mathcal{W}^{u}\loc(x)$ to $\mathcal{W}^{u}(q)$.
Since $q\in \widehat{J}^{rsc}_n(x)$, we have $d(x,q)=O(\delta_n)$.
Let $x'$ be in the boundary of $\mathcal{W}^u(x,\delta_{n})$, then $d(x,x')=\delta_n$. Let $q'\coloneqq h(x')$. From Proposition \ref{Prop:UnstableHolonomyBounds} we get that $d(p',q')=O(\delta_n)$. Thus, for $k_0$ large enough we have that $\delta_{n+k_0}\leq d(q,q')\leq \delta_{n-k_0} $. The result follows from this.
\end{proof}
\bibliographystyle{alpha}
\bibliography{references}

\begin{thebibliography}{RHRHU08}

\bibitem[AH94]{aoki1994topological}
N.~Aoki and K.~Hiraide.
\newblock {\em Topological Theory of Dynamical Systems: Recent Advances}.
\newblock ISSN. Elsevier Science, 1994.

\bibitem[Ano67]{Anosov97}
D.~V. Anosov.
\newblock Geodesic flows on closed {R}iemannian manifolds of negative curvature.
\newblock {\em Trudy Mat. Inst. Steklov.}, 90:209, 1967.

\bibitem[BP74]{BP74}
M.~I. Brin and Ja.~B. Pesin.
\newblock Partially hyperbolic dynamical systems.
\newblock {\em Izv. Akad. Nauk SSSR Ser. Mat.}, 38:170--212, 1974.

\bibitem[BW10]{BW10}
Keith Burns and Amie Wilkinson.
\newblock On the ergodicity of partially hyperbolic systems.
\newblock {\em Annals of Mathematics}, 171(1):451--489, 2010.

\bibitem[CK12]{climenhaga2012measuretheorydynamicaleyes}
Vaughn Climenhaga and Anatole Katok.
\newblock Measure theory through dynamical eyes, 2012.
\newblock https://arxiv.org/abs/1208.4550.

\bibitem[DSL18]{LDS_Limit_theorems}
Jacopo De~Simoi and Carlangelo Liverani.
\newblock Limit theorems for fast--slow partially hyperbolic systems.
\newblock {\em Inventiones mathematicae}, 213(3):811--1016, 2018.

\bibitem[Fol13]{folland2013real}
G.B. Folland.
\newblock {\em Real Analysis: Modern Techniques and Their Applications}.
\newblock Pure and Applied Mathematics: A Wiley Series of Texts, Monographs and Tracts. Wiley, 2013.

\bibitem[GPS94]{GraysonPughShub1994}
Matthew Grayson, Charles Pugh, and Michael Shub.
\newblock Stably ergodic diffeomorphisms.
\newblock {\em Annals of Mathematics}, 140(2):295--329, 1994.

\bibitem[He17]{Baolin}
Baolin He.
\newblock Accessibility of partially hyperbolic endomorphisms with 1d center-bundles.
\newblock {\em Journal of Applied Analysis \& Computation}, 7(1):334--345, 2017.

\bibitem[HH22]{Layne-Andy}
Layne Hall and Andy Hammerlindl.
\newblock Classification of partially hyperbolic surface endomorphisms.
\newblock {\em Geometriae Dedicata}, 216(3):29, 2022.

\bibitem[Hop39]{EHopf39}
Eberhard Hopf.
\newblock Statistik der geod\"{a}tischen {L}inien in {M}annigfaltigkeiten negativer {K}r\"{u}mmung.
\newblock {\em Ber. Verh. S\"{a}chs. Akad. Wiss. Leipzig Math.-Phys. Kl.}, 91:261--304, 1939.

\bibitem[HPS77]{HSP}
M.W. Hirsch, C.C. Pugh, and M.~Shub.
\newblock {\em Invariant Manifolds}.
\newblock Lecture Notes in Mathematics. Springer, 1977.

\bibitem[MnP75]{Mane}
Ricardo Ma\~{n}\'{e} and Charles Pugh.
\newblock Stability of endomorphisms.
\newblock In {\em Dynamical systems---{W}arwick 1974 ({P}roc. {S}ympos. {A}ppl. {T}opology and {D}ynamical {S}ystems, {U}niv. {W}arwick, {C}oventry, 1973/1974; presented to {E}. {C}. {Z}eeman on his fiftieth birthday)}, pages 175--184. Lecture Notes in Math., Vol. 468, 1975.

\bibitem[Prz76]{przytycki}
Feliks Przytycki.
\newblock Anosov endomorphisms.
\newblock {\em Studia Math.}, 58(3):249--285, 1976.

\bibitem[PS96]{PughShub96}
Charles Pugh and Michael Shub.
\newblock Stable ergodicity and partial hyperbolicity.
\newblock In {\em International {C}onference on {D}ynamical {S}ystems ({M}ontevideo, 1995)}, volume 362 of {\em Pitman Res. Notes Math. Ser.}, pages 182--187. Longman, Harlow, 1996.

\bibitem[PS00]{PS2000}
Charles Pugh and Michael Shub.
\newblock Stable ergodicity and julienne quasi-conformality.
\newblock {\em J. Eur. Math. Soc. (JEMS)}, 2(1):1--52, 2000.

\bibitem[Pug13]{pugh2013real}
C.C. Pugh.
\newblock {\em Real Mathematical Analysis}.
\newblock Undergraduate Texts in Mathematics. Springer New York, 2013.

\bibitem[PV19]{ponce2019introduction}
G.~Ponce and R.~Var{\~a}o.
\newblock {\em An Introduction to the Kolmogorov--Bernoulli Equivalence}.
\newblock SpringerBriefs in Mathematics. Springer International Publishing, 2019.

\bibitem[QXZ09]{qian2009smooth}
M.~Qian, J.S. Xie, and S.~Zhu.
\newblock {\em Smooth Ergodic Theory for Endomorphisms}.
\newblock Lecture Notes in Mathematics. Springer Berlin Heidelberg, 2009.

\bibitem[RHRHU08]{RHRHU08}
F.~Rodriguez~Hertz, M.~A. Rodriguez~Hertz, and R.~Ures.
\newblock Accessibility and stable ergodicity for partially hyperbolic diffeomorphisms with 1{D}-center bundle.
\newblock {\em Invent. Math.}, 172(2):353--381, 2008.

\bibitem[Rok67]{Rokhlin_1967}
V~A Rokhlin.
\newblock Lectures on the entropy theory of measure-preserving transformations.
\newblock {\em Russian Mathematical Surveys}, 22(5):1, oct 1967.

\bibitem[Tsu00]{Tsujii}
Masato Tsujii.
\newblock Piecewise expanding maps on the plane with singular ergodic properties.
\newblock {\em Ergodic Theory and Dynamical Systems}, 20(6):1851--1857, December 2000.

\end{thebibliography}

\end{document}